\documentclass[11pt]{amsart}
\usepackage{amsfonts,amssymb,amscd,amstext,mathrsfs,mathtools}
\usepackage[utf8]{inputenc}
\usepackage{hyperref}
\usepackage{verbatim}

\usepackage{graphics}
\usepackage{graphicx}

\usepackage{times}
\usepackage{enumerate}
\usepackage[up,bf]{caption}
\usepackage{color}
\usepackage{t1enc}

\input xy
\xyoption{all}

\numberwithin{equation}{section}

\newtheorem{theorem}{Theorem}[section] 
\newtheorem{proposition}[theorem]{Proposition} 
 
\newtheorem{lemma}[theorem]{Lemma} 
\newtheorem{corollary}[theorem]{Corollary} 

\theoremstyle{definition} 
\newtheorem{definition}[theorem]{Definition} 
\newtheorem{remark}[theorem]{Remark} 
 
\newtheorem{problem}[theorem]{Problem} 
\newtheorem{example}[theorem]{Example} 
 
\newcommand\Tcal{\mathcal{T}} 
 
\newcommand\Vcal{\mathcal{V}} 

\newcommand\Cscr{\mathscr{C}} 
\newcommand\Dscr{\mathscr{D}} 
\newcommand\Escr{\mathscr{E}}

\newcommand\Jscr{\mathscr{J}}
 
\newcommand\Oscr{\mathscr{O}}

\newcommand\C{\mathbb{C}}

\newcommand\N{\mathbb{N}} 
 
\newcommand\R{\mathbb{R}}

\newcommand\U{\mathbb{U}} 
\newcommand\Z{\mathbb{Z}}

\newcommand\igot{\mathfrak{i}}

\renewcommand\igot{\mathfrak{i}}

\renewcommand\imath{\igot}

\newcommand\hra{\hookrightarrow} 
\newcommand\lra{\longrightarrow} 
 
\newcommand\longhookrightarrow{\ensuremath{\lhook\joinrel\relbar\joinrel\rightarrow}} 

\newcommand\wt{\widetilde} 
\newcommand\wh{\widehat} 
\newcommand\di{\partial} 
\newcommand\dibar{\overline\partial}

\newcommand\dist{\mathrm{dist}}

\newcommand\Id{\mathrm{Id}}

\newcommand\supp{\mathrm{supp}}

\newcommand\Pic{\mathrm{Pic}}

\def\dist{\mathrm{dist}}

\def\Ell1{\mathrm{Ell_1}} 
\def\CEll1{\mathrm{CEll_1}}

\def\Jst{J_{\mathrm{st}}} 

\begin{document} 

\title{The Oka principle for tame families of Stein manifolds} 

\author{Franc Forstneri\v c and {\'A}lfhei{\dh}ur Edda Sigur{\dh}ard{\'o}ttir}

\address{Franc Forstneri\v c, Faculty of Mathematics and Physics, University of Ljubljana, Jadranska 19, SI--1000 Ljubljana, Slovenia}

\address{Franc Forstneri\v c, Institute of Mathematics, Physics and Mechanics, Jadranska 19, SI--1000 Ljubljana, Slovenia}

\email{franc.forstneric@fmf.uni-lj.si}

\address{Álfheiður Edda Sigurðardóttir, 
Institute of Mathematics, Physics and Mechanics, Jadranska 19, SI--1000 Ljubljana, Slovenia} 


\email{edda.sigurdardottir@imfm.si}


\date{29 October 2025}


\begin{abstract}
Let $X$ be a smooth open manifold of even dimension, 
$T$ be a topological space, and $\Jscr=\{J_t\}_{t\in T}$ be a
continuous family of smooth integrable Stein structures on $X$.
Under suitable additional assumptions on $T$ and $\Jscr$, we prove 
an Oka principle for continuous families of maps from the family of 
Stein manifolds $(X,J_t)$, $t\in T$, to any Oka manifold,
showing that every family of continuous maps is homotopic
to a family of $J_t$-holomorphic maps depending continuously on $t$. 
We also prove the Oka--Weil theorem for sections of
$\Jscr$-holomorphic vector bundles on $Z=T\times X$ and
the Oka principle for isomorphism classes of such bundles.  
The assumption on the family $\Jscr$ is that the $J_t$-convex hulls 
of any compact set in $X$ are upper semicontinuous with respect 
to $t\in T$; such a family is said to be tame. 
For suitable parameter spaces $T$, we characterise 
tameness by the existence of a continuous 
family $\rho_t:X\to \R_+=[0,+\infty)$, $t\in T$, of strongly 
$J_t$-plurisubharmonic exhaustion functions on $X$. 
Every family of complex structures on an open orientable surface is tame.
We give an example of a nontame smooth family 
of Stein structures $J_t$ on $\R^{2n}$ $(t\in \R,\ n>1)$ such that 
$(\R^{2n},J_t)$ is biholomorphic 
to $\C^n$ for every $t\in\R$. We show that the 
Oka principle fails on any nontame family.
\end{abstract}

\maketitle 

\setcounter{tocdepth}{1}
\tableofcontents

%
%
%

%
%
\section{Introduction}\label{sec:intro} 

Let $X$ be a smooth manifold of dimension $2n\ge 2$.
An almost complex structure $J$ on $X$ is an endomorphism 
$J:TX\to TX$ of its tangent bundle satisfying $J^2=-\Id$. 
When $n=1$, i.e., $X$ is a smooth surface, every such $J$ 
of local H\"older class $\Cscr^\alpha$, $0<\alpha<1$, 
determines on $X$ the structure of a Riemann surface
\cite[Theorem 5.3.4]{AstalaIwaniecMartin2009}; 
if $X$ is an open surface then $(X,J)$ is a Stein manifold 
according to Behnke and Stein \cite{BehnkeStein1949}.
In \cite{Forstneric2024Runge} the first named author
showed that, under suitable regularity assumptions on 
the parameter space $T$ and on a family $\Jscr=\{J_t\}_{t\in T}$
of complex structures on a smooth open surface $X$, 
the Oka principle holds for families
of $J_t$-holomorphic maps from $X$ to any 
Oka manifold $Y$, with continuous or smooth 
dependence on $t\in T$ and with 
approximation on suitable families of Runge subsets of $X$. 
The notion of an Oka manifold 
(see \cite{Forstneric2009CR}, \cite[Sect.\ 5.4]{Forstneric2017E}, 
and \cite{Forstneric2023Indag}) developed from the classical 
Oka--Grauert--Gromov principle 
\cite{Oka1939,Grauert1958MA,Gromov1989}.

In this paper we study the mapping problem for families 
of Stein structures on smooth manifolds of dimension 
$2n\ge 4$. Integrability is then a nontrivial condition; see 
Section \ref{sec:almostcomplex}. 
However, this is not the only new issue.
The construction in \cite{Forstneric2024Runge}
strongly uses the fact that the holomorphic hull
of a compact set in a smooth surface $X$ is independent
of the choice of the complex structure on $X$.
This is no longer the case on higher dimensional manifolds.
In Theorem \ref{th:nontame} we give an example
of a smooth family of integrable Stein structures 
$\{J_t\}_{t\in\R}$ on $\R^{2n}$ for any $n>1$ such that 
$(\R^{2n},J_t)$ is biholomorphic to 
$\C^n$ for every $t\in\R$ but the $J_t$-convex hulls 
of the closed ball explode when $t\in \R\setminus\{0\}$ approaches $0$.
This phenomenon excludes the possibility of any 
reasonable analysis of global analytic problems. 
Motivated by this example, we introduce a tameness condition 
on a family of Stein structures $\{J_t\}_{t\in T}$ 
on a smooth manifold $X$ which excludes this type of pathology. 
Such a family is said to be tame if the $J_t$-convex hulls of any 
compact set in $X$ are upper semicontinuous with respect to $t\in T$;
see Definition \ref{def:tame}. 
Tameness is characterised by the existence of a continuous 
family of strongly $J_t$-plurisubharmonic exhaustion functions 
$\rho_t:X\to\R_+$; see Theorem \ref{th:tame}.
Every family of Riemann surface structures is tame. 
We give several examples of tame families 
of Stein structures on higher dimensional manifolds.
%

The following Oka principle is a special case of our main result, 
Theorem \ref{th:Oka}. 

%
%
\begin{theorem}
Assume that $T$ is a finite CW complex, $X$ is a smooth manifold, 
$\Jscr=\{J_t\}_{t\in T}$ is a tame family of smooth Stein structures on $X$
depending continuously on $t$, and $Y$ is an Oka manifold.
Then, every continuous map $f:Z=T\times X\to Y$ is homotopic
to a $\Jscr$-holomorphic map $F:Z\to Y$, i.e. such that 
$F(t,\cdotp):X\to Y$ is $J_t$-holomorphic for every $t\in T$
and continuous in $t$.
If $f$ is $\Jscr$-holomorphic on a neighbourhood
of a closed subset $K\subset Z$ with proper projection $K\to T$
and $J_t$-convex fibres $K_t$ $(t\in T)$, 
then $F$ can be chosen to approximate $f$ in
the fine topology on $K$.
\end{theorem}

The special case when $Y$ is the complex number field
$\C$ is the Oka--Weil theorem for such families; see
Theorem \ref{th:OkaWeil}. We show 
in Corollary \ref{cor:extension} that the Oka principle 
fails on any nontame family, so tameness is a necessary and
sufficient condition for the Oka principle. 
The Oka--Weil theorem is also proved for sections of fibrewise holomorphic 
vector bundles on tame families of Stein structures; see
Theorem \ref{th:OkaWeilB}. This is used to obtain
global solutions of the $\dibar$-equation for fibrewise 
smooth $(p,q)$-forms in all bidigrees, 
see Theorem \ref{th:dibarGlobal}. We also prove the Oka principle 
for the classification of complex vector bundles on such families,
extending the classical results of Oka \cite{Oka1939} and Grauert
\cite{Grauert1958MA}; see Theorems \ref{th:OPLB} and 
\ref{th:OPvectorbundles} . 
Our results open a new direction in modern Oka theory. 

An important ingredient in the proofs 
is a theorem of Hamilton \cite{Hamilton1977},  
also called the global Newlander--Nirenberg theorem, 
on representing small integrable deformations of the
complex structure on the closure of a smoothly 
bounded, relatively compact, strongly pseudoconvex domain 
$\Omega$ in a Stein manifold $X$ 
by small smooth deformations of $\Omega$ in $X$.
We need a version of this result with continuous dependence on
parameters; see Theorem \ref{th:Hamilton}, which 
is obtained from the proof of Hamilton's theorem 
by Greene and Krantz \cite[Theorem 1.13]{GreeneKrantz1982}.
Unlike the original proof and its improvements 
\cite{GanGong2024,GongShi2024}, which use the Nash--Moser technique, 
the proof in \cite{GreeneKrantz1982} 
is based on stability of the canonical (Kohn) solution of the $\dibar$-equation 
with respect to perturbations of the complex structure, obtained in 
\cite[Theorem 3.10]{GreeneKrantz1982} by following the pioneering work 
of Kohn \cite{Kohn1963,Kohn1964} on the $\dibar$-Neumann problem. 
A special case of Hamilton's theorem with parameters for smoothly 
bounded domains in Riemann surfaces, and under considerably lower regularity assumptions on the family of complex structures, 
was obtained by the first named author in 
\cite[Theorem 4.3]{Forstneric2024Runge} using the Beltrami equation.

%
%
%
%
\section{Almost complex structures and integrability}
\label{sec:almostcomplex}

In this section we recall the relevant background concerning 
almost complex structures. 

Let $X$ be a smooth manifold of real dimension $2n$. 
An almost complex structure $J$ on $X$ is an endomorphism 
$J:TX\to TX$ of its tangent bundle satisfying $J^2=-\Id$. 
Every point $x_0\in X$ has an open coordinate 
neighbourhood $U\subset X$ such that $TX|_U\cong U\times \R^{2n}$
and $J:TX|_U\to TX|_U$ is given by $(x,\xi)\mapsto (x,A(x)\xi)$, 
where the matrix $A(x)\in GL_{2n}(\R)$ satisfies $A(x)^2=-I$ 
with $I\in  GL_{2n}(\R)$ the identity matrix. 
We say that $J$ is of class $\Cscr^k$ if its matrix $A(x)$ 
in any smooth local coordinate on $U\subset X$ is a $\Cscr^k$ map
$U\to GL_{2n}(\R)$. Similarly one defines (local) H\"older classes
$\Cscr^{(k,\alpha)}$ with $k\in \Z_+$ and $0<\alpha<1$;
see \cite[Sect. 4.1]{GilbargTrudinger1983}.
An almost complex structure $J$ extends to an endomorphism of the 
complexified tangent bundle $\C TX=TX \otimes_\R \C$. 
Since $J_x^2=-\Id$ holds for every $x\in X$,
$J$ induces a decomposition $\C TX=H\oplus \overline H$ 
into a direct sum of complex subbundles of rank $n$ 
whose fibres $H_x$ and $\overline H_x$ over $x\in X$ are, 
respectively, the $+\imath=\sqrt{-1}$ and $-\imath$ eigenspaces of 
$J_x$ on $\C T_x X$. This gives complex vector bundle projections
$\pi_{1,0}:\C TX\to H$ and $\pi_{0,1}:\C TX\to \overline H$ satisfying 
\begin{equation}\label{eq:projections}
	\pi_{1,0}=\bar \pi_{0,1},
	\quad
	\pi_{1,0}+\pi_{0,1}=\Id,
	\quad
	\pi_{1,0}\circ \pi_{0,1}=0=\pi_{0,1}\circ \pi_{1,0}.
\end{equation}
Conversely, a pair of such projections determines an almost complex 
structure $J$ on $X$. Note that $\pi_{1,0}$ and $\pi_{0,1}$ are as smooth 
as $J$. An almost complex structure $J$ of class $\Cscr^1$ on $X$ is 
said to be (formally) {\em integrable} if the subbundle $H=\pi_{1,0}(\C TX)$
satisfies the commutator condition $[H,H]\subset H$ for its sections, 
which are called vector fields of type $(1,0)$. Every such vector field
is of the form $v-\imath Jv$ where $v$ is a real vector field on $X$.
If $n=1$ then the commutator condition is void, but 
integrability is a nontrivial condition when $n\ge 2$. 
For later reference, we state the following precise version of 
the Newlander--Nirenberg integrability theorem.

%
%
\begin{theorem}\label{th:integrable}
If $X$ is a smooth manifold of dimension $2n$ and $J$ is an integrable 
almost complex structure on $X$ of local H\"older class $\Cscr^{(k,\alpha)}$, 
with $k\ge 1$ an integer (or $k\ge 0$ when $X$ is a surface) 
and $0<\alpha<1$,  then every point $x_0\in X$ has a neighbourhood $U\subset X$ with a $J$-holomorphic coordinate map $z:U\to\C^n$ 
of class $\Cscr^{(k+1,\alpha)}$. Thus, $(X,J)$ is a complex manifold,
and the smooth structure on $X$ determined by $J$ is 
$\Cscr^{(k+1,\alpha)}$ compatible with the given smooth structure.
\end{theorem} 

This result has a complex genesis. For surfaces ($n=1$),
see Korn \cite{Korn1914}, Lichtenstein \cite{Lichtenstein1916},
Chern \cite{Chern1955}, and Astala et al. 
\cite[Theorem 5.3.4]{AstalaIwaniecMartin2009}.
For $n>1$ the result is due to Newlander and Nirenberg 
\cite{NewlanderNirenberg1957} under stronger regularity
assumptions. Improvements were given 
by Nijenhuis and Woolf \cite{NijenhuisWoolf1963},
Kohn \cite[Theorem 12.1]{Kohn1963}, 
Malgrange \cite{Malgrange1969}, 
Webster \cite[Theorem 3.1]{Webster1989}, 
Treves \cite{Treves1992}, and possibly others.
(See also Nirenberg \cite{Nirenberg1973} and 
H\"ormander \cite[Sect.\ 5.7]{Hormander1990}.) 
The last statement concerning the compatibility
of smooth structures follows from the fact 
that the inverse of a diffeomorphism of local H\"older class 
$\Cscr^{(k,\alpha)}$ with $k\ge 1$ is of the same class; 
see Norton \cite{Norton1986} and 
Bojarski et al.\ \cite[Theorem 2.1]{BojarskiAll2005}.
A 1-parametric version of the Newlander--Nirenberg 
theorem was proved by Gong \cite{Gong2020}.

Denote by $\Lambda^l(\C T^*X)$ the $l$-th exterior power of 
the complexified cotangent bundle $\C T^*X$.
Its sections are complex differential $l$-forms on $X$. 
The projections $\pi_{1,0}$ and $\pi_{0,1}$ in \eqref{eq:projections}
give rise to projections 
$\pi_{p,q}:\Lambda^l(\C T^*X)\to \Lambda^{p,q}(\C T^*X)$
onto complex vector subbundles of $\Lambda^l(\C T^*X)$
for $0\le p,q\le n$, with $p+q=l \in \{1,\ldots,2n\}$, such that 
$\oplus_{p+q=l} \Lambda^{p,q}(\C T^*X) = \Lambda^l(\C T^*X)$.
Sections of $\Lambda^{p,q}(\C T^*X)$ are 
differential forms of bidegree $(p,q)$ with respect to the complex 
structure $J$ on $X$. Assuming that $J$ is of class 
$j\in\{0,1,\ldots,\infty\}$, these subbundles are also of class $\Cscr^j$. 
Let $\Dscr^{p,q}_j(X)$ denote the space of $(p,q)$-forms of class 
$\Cscr^j$ on $X$, with $\Dscr^{p,q}(X)=\Dscr^{p,q}_\infty(X)$, 
and let $d$ be the exterior differential on $X$.
We have the operators $\di=\di_J$ and $\dibar=\dibar_J$ defined by 
\[
	\di=\pi_{p+1,q}\circ d: \Dscr^{p,q}_j(X)\to \Dscr^{p+1,q}_{j-1}(X),
	\quad
	\dibar=\pi_{p,q+1}\circ d:\Dscr^{p,q}_j(X)\to \Dscr^{p,q+1}_{j-1}(X).
\]
Integrability of $J$ is equivalent to each of the conditions 
$\di^2=0$, $\dibar^2=0$, and $d=\di+\dibar$ 
(see \cite[Proposition 1.2.1]{FollandKohn1972}).
If $J$ is integrable then the kernel of the operator $\dibar$ (resp.\ $\di$) 
on functions are precisely the $J$-holomorphic 
(resp.\ the $J$-antiholomorphic) functions.
We also have the conjugate differential $d^c=d^c_J=\imath(\dibar-\di)$ 
and the operator $dd^c=2\imath \di\dibar$.
For a $\Cscr^2$ function $\rho:X\to\R$, 
$dd^c\rho$ is a $(1,1)$-form called the {\em Levi form} of $\rho$.
A function $\rho$ is said to be (strongly) {\em $J$-plurisubharmonic} if 
$dd^c\rho\ge 0$ (resp.\ $dd^c\rho > 0$), in the sense that 
for any $x\in X$ and $0\ne v\in T_x X$ we have 
$\langle dd^c\rho(x),v\wedge Jv\rangle \ge 0$ (resp.\ $>0$);
see \cite[Eq. (1.39), p.\ 30]{Forstneric2017E}.
A complex manifold $(X,J)$ is a Stein manifold if and only if it admits 
a strongly $J$-plurisubharmonic exhaustion function $\rho:X\to\R_+$
(see Grauert \cite{Grauert1958AM}).
A necessary and sufficient topological condition for the existence of 
an integrable Stein structure on a smooth manifold 
of dimension $2n\ge 6$ was given by Eliashberg
\cite{Eliashberg1990,CieliebakEliashberg2012}. The situation 
is more complicated on manifolds of dimension $4$;
see Gompf \cite{Gompf1998,Gompf2005} and 
\cite[Chap.\ 10]{Forstneric2017E}.

A domain $D\Subset X$ with $\Cscr^2$ boundary is said to be 
{\em strongly pseudoconvex}, or {\em strongly $J$-pseudoconvex} 
if we wish to emphasise the choice of the complex structure $J$,  
if it admits a defining function $\rho:U\to\R$ on a neighbourhood $U$ 
of $\bar D$ such that $D=\{\rho<0\}$, $d\rho\ne 0$ 
on $bD=\{\rho=0\}$, and $dd^c\rho(x)>0$ for every $x\in bD$. 
See Krantz \cite{Krantz2001} for the basic theory of such domains.

Fix a smooth Riemannian metric $g$ on $X$.
Such $g$ extends to a field of $\C$-bilinear forms on the complexified
tangent spaces $\C T_xX$, $x\in X$. Given an almost complex 
structure $J$ on $X$ determined by the projections \eqref{eq:projections}, 
write a vector $u\in \C TX$ in the form
$u=u_{1,0}+u_{0,1}$ where $u_{1,0}=\pi_{1,0}(u)$ and $u_{0,1}=\pi_{0,1}(u)$.
Then, $g$ and $J$ determine a field of inner products on 
the fibres of $\C TX$ by
\begin{equation}\label{eq:inner}
	\langle u,v\rangle_{\! J}
	= g(u_{1,0},\overline{v_{1,0}}) + g(u_{0,1},\overline{v_{0,1}}),	
	\quad u,v\in \C T_xX,\ x\in X
\end{equation}
(cf.\ \cite[p.\ 8]{FollandKohn1972}). 
This inner product is $J$-hermitian on the subbundle $H\subset \C TX$ 
on which $J=\imath$, $J$-antihermitian on 
the conjugate subbundle $\overline H\subset \C TX$ 
on which $J=-\imath$, the subbundles $H$ and $\overline H$ 
are $\langle \cdotp,\cdotp\rangle_J$-orthogonal, and 
for every $u\in \C TX$ we have
\[
	\|u\|^2=\langle u,u\rangle_{\! J} =
	\|\Re u_{1,0}\|_g^2 + \|\Im u_{1,0}\|_g^2 
	+ \|\Re u_{0,1}\|_g^2 + \|\Im u_{0,1}\|_g^2. 
\] 
Here, $\Re$ and $\Im$ denote the real and imaginary part.
By duality and multilinear algebra, the field of inner products
$\langle \cdotp,\cdotp\rangle_{\! J}$  in 
\eqref{eq:inner} extends to the bundles $\Lambda^{p,q}(\C T^*X)$.
If $J$ is of class $\Cscr^j$ then so is 
$\langle \cdotp,\cdotp\rangle_{\! J}$,
and if an almost complex structure $J'$ on $X$ is $\Cscr^j$ close to $J$ 
then $\langle \cdotp,\cdotp\rangle_{\! J'}$ is $\Cscr^j$ close to
$\langle \cdotp,\cdotp\rangle_{\! J}$. 
Given a domain $D\Subset X$ with $\Cscr^1$ boundary, 
we have an inner product of forms $\phi,\psi\in \Dscr^{p,q}_j(\bar D)$ 
given by
\begin{equation}\label{eq:Jproduct}
	(\phi,\psi)_J = \int_D \langle \phi,\psi \rangle_{\! J}\, dV
\end{equation}
where $dV$ is the volume form on $X$ determined by $g$. 
If $\{J_t\}_{t\in T}$ is a continuous family of almost complex 
structures on $X$ then the inner products $(\cdotp,\cdotp)_{J_t}$ also vary  
continuously, and the $L^2$ norms $\|\phi\|_{J_t}^2 = (\phi,\phi)_{J_t}$ 
are comparable for $t$ in any compact subset of $T$.

We shall be dealing with families $\Jscr=\{J_t\}_{t\in T}$ of integrable 
complex structures on given smooth manifold $X$, where $T$ is 
a topological space whose precise properties will be specified in 
the individual results. A continuous map $f:Z=T\times X\to Y$ to a complex
manifold $Y$ is said to be {\em $\Jscr$-holomorphic} if the map 
$f(t,\cdotp):X\to Y$ is $J_t$-holomorphic for every $t\in T$.
Such a family $\Jscr$ is said to be of class $\Cscr^{0,k}$, 
where $k\in \{0,1,\ldots,\infty\}$, if $J_t$ admits partial
derivatives of order up to $k$ in the space variable $x\in X$ 
and these derivative depend continuously on $t\in T$.
If $0<\alpha<1$, we say that $J_t$ is of local 
H\"older class $\Cscr^{(k,\alpha)}$ on $X$ 
and depends continuously on $t$ if for every 
smoothly bounded relatively compact domain $\Omega\Subset X$, 
$J_t|_{T\Omega} \in \mathrm{Hom}^{(k,\alpha)}(T\Omega,T\Omega)$
depends continuously on $t\in T$ as an element of this space. 
The analogous definition applies to functions or maps 
on $T\times X$. If $T$ is a $\Cscr^l$ manifold then a function 
is of class $\Cscr^{l,k}(T\times X)$ if it has
$l$ derivatives in $t\in T$ followed by $k$ derivatives in $x\in X$,
and these derivatives are continuous.  Similarly one defines
the H\"older classes $\Cscr^{l,(k,\alpha)}$.

%
%
%
%
\section{A theorem of Hamilton for families of complex structures}
\label{sec:Hamilton}

The main result of this section, Theorem \ref{th:Hamilton},
is a version of Hamilton's theorem \cite{Hamilton1977}
(also called the global Newlander--Nirenberg theorem) 
for a family of smooth integrable complex structures
on a compact strongly pseudoconvex domain in a Stein manifold.
It is used in the proof of all main results in the paper.
Its proof uses stability of Kohn's solution of the
$\dibar$-equation on such domains, obtained by 
Greene and Krantz \cite{GreeneKrantz1982} and based
on the work of Kohn \cite{Kohn1963,Kohn1964}; 
see Theorem \ref{th:dibar}.

Assume that $(X,J)$ is a Stein manifold and $D\Subset X$ is 
a relatively compact, smoothly bounded, strongly $J$-pseudoconvex 
domain. A theorem of Hamilton \cite{Hamilton1977} 
says that for every sufficiently small smooth integrable 
deformation $J'$ of the complex structure $J$ on 
$\bar D$ there is a smooth diffeomorphism $F:\bar D\to F(\bar D)\subset X$,  
close to the identity map on $\bar D$, such that $J'=F^*J$ is the pullback 
of $J$ by $F$. Equivalently, the map $F:D\to F(D)$ is biholomorphic
from $(D,J')$ onto $(F(D),J)$. 
(Hamilton's result applies to a wider class of domains
but we shall restrict the attention to this case.)  
The proof in \cite{Hamilton1977} 
is nonlinear in nature and uses the Nash--Moser technique.
Improvements in terms of the required regularity 
of the almost complex structure and of the boundary of the domain 
were obtained by Gan and Gong \cite{GanGong2024}, Shi \cite{Shi2025} 
(for strongly pseudoconvex domains in $\C^n$),  
and Gong and Shi \cite{GongShi2024}.
It was shown by Hill \cite{Hill1989} that the result 
fails in general on domains with Levi degenerate boundaries.

A simpler proof of Hamilton's theorem on strongly pseudoconvex
domains was given 
by Greene and Krantz \cite[Theorem 1.13]{GreeneKrantz1982} by  
using the $\dibar$--Neumann method of Kohn 
\cite{Kohn1963,Kohn1964,FollandKohn1972}
for solving the $\dibar$-equation.
Their approach, together with stability results 
for Kohn's solutions of the $\dibar$-equation
with respect to a family of complex structures 
(see \cite[Sect.\ 3]{GreeneKrantz1982} and Theorem \ref{th:dibar}), 
will be used to give the following parametric version
of Hamilton's theorem. 

%
%
\begin{theorem}\label{th:Hamilton}
Let $(X,J)$ be a Stein manifold of complex dimension $n$.
Let $k\ge 1$ and $r\ge 2k+2n+9$ be integers, 
and let $D\Subset X$ be a relatively compact, 
strongly pseudoconvex domain with boundary $bD$ of class
$\Cscr^r$. Assume that $T$ is a topological space and 
$\Jscr=\{J_t\}_{t\in T}$ is a family of integrable 
complex structures on $\bar D$ of class $\Cscr^{0,r}(\bar D)$ 
such that for some $t_0\in T$, $J_{t_0}$ is the restriction of $J$ to $\bar D$. 
Then there exist a neighbourhood $T_0\subset T$ of $t_0$ and a 
family of diffeomorphisms $F_t:D\to D_t=F_t(D)\subset X$ 
in $\Cscr^k(D,X)$, depending continuously on $t\in T_0$,  
such that $F_t$ is a biholomorphic map from $(D,J_t)$
onto $(D_t,J_{t_0})$ for every $t\in T_0$ and 
$F_{t_0}$ is the identity on $D$. If $bD\in \Cscr^\infty$ and 
$\Jscr$ is of class $\Cscr^{0,\infty}(T\times \bar D)$
then the family $\{F_t\}_{t\in T_0}$ can be chosen 
to be of class $\Cscr^{0,\infty}$ on $T_0\times \bar D$.
\end{theorem}

Theorem \ref{th:Hamilton} is likely not optimal in terms of regularity. 
For relatively compact domains in open 
Riemann surfaces, a more precise result 
\cite[Theorem 4.3]{Forstneric2024Runge} 
was obtained via the Beltrami equation.

We begin with preliminaries. Choose a smooth 
Riemannian metric $g$ on $X$ and let $dV$ be the associated volume
measure. Fix a relatively compact domain $D\Subset X$ with $\Cscr^1$ 
boundary. Let $L^2(D)$ denote the space of measurable functions $f$ 
on $D$ with $\|f\|^2_{L^2(D)} = \int_D |f|^2 dV<+\infty$.
For $s\in \Z_+$ we denote by $H_s(D)=W^{s,2}(D)$ 
the Sobolev (Hilbert) space of functions on $D$ whose derivatives 
of order up to $s$ belong to $L^2(D)$. 
In particular, $H_0(D)=L^2(D)$.
(For a discussion of Sobolev spaces, see Adams \cite{AdamsRA1975} or 
Folland and Kohn \cite[Appendix]{FollandKohn1972}.) 
When $X=\R^{N}$ with the Euclidean metric, 
the norm on $H_s(D)$ is given by 
$\|f\|_{H_s(D)}^2=\sum_{|\alpha|\le s} \|D^\alpha f\|_{L^2(D)}^2$, where 
$D^\alpha$ for a multiindex $\alpha\in\Z_+^{N}$ denotes a partial derivative 
of order $|\alpha|$ on $\R^N$.
On a smooth manifold $X$ we introduce these and 
other norms mentioned in the sequel by using a finite covering 
of $\bar D$ by smooth charts; see 
\cite[Appendix, p. 122]{FollandKohn1972}.
By $\Cscr^s(D)$ we denote the Banach space of functions 
having continuous bounded partial derivatives of order $\le s$, with 
$\|f\|_{\Cscr^s(D)}=\sum_{|\alpha|\le s} \sup_{x\in D}|D^\alpha f(x)|$
when $D\subset\R^N$. 
Given an integrable almost complex structure $J$ on $X$, 
we have the induced metrics on the vector bundles 
$\Lambda^{p,q}(\C T^*X)$ of differential $(p,q)$-forms
(see Section \ref{sec:almostcomplex}).
We denote by $H_{s}^{p,q}(D,J)$ the Sobolev space $W^{s,2}(D)$ 
of $(p,q)$-forms on $D$ with respect to $J$, endowed
with the inner product \eqref{eq:Jproduct}.

The following result is \cite[Theorem 3.10]{GreeneKrantz1982} 
by Greene and Krantz. The regularity statements for a single
complex structure are due to Kohn \cite{Kohn1963};
see also \cite[Proposition 3.1.15, p.\ 52]{FollandKohn1972}.

%
%
\begin{theorem} 
\label{th:dibar}
Assume that $X$ is a smooth Riemannian manifold of real dimension 
$2n\ge 2$, $s\ge 1$ is an integer, $D\Subset X$ is
a relatively compact domain with boundary of class 
$\Cscr^{2s+5}$, $T$ is a topological space, 
and $\Jscr=\{J_t\}_{t\in T}$ is a continuous 
family of integrable Stein structures 
of class $\Cscr^{2s+5}$ on $\bar D$ (i.e., $\Jscr$ is of class 
$\Cscr^{0,2s+5}$ on $T\times \bar D$)
such that $D$ is strongly $J_t$-pseudoconvex with 
Stein interior for every $t\in T$. 
Then the following assertions hold.
\begin{enumerate}[\rm (a)]
\item 
For every $\alpha\in H^{p,q}_{0}(D,J_t)$ $(p\ge 0,\ q\ge 1,\ t\in T)$ with 
$\dibar_{J_t} \alpha=0$ there is a unique (Kohn) solution 
$K_{t}\alpha \in H_{0}^{p,q-1}(D,J_t)$ of the equation 
$\dibar_{J_t} (K_{t}\alpha) =\alpha$ 
satisfying $K_{t}\alpha \perp \ker(\dibar_{J_t})$ with respect to
the inner product $(\cdotp,\cdotp)_{J_t}$ given by \eqref{eq:Jproduct}.
\item  
If $\alpha\in H^{p,q}_{s}(D,J_t)$ then 
$K_{t}\alpha\in H_{s}^{p,q-1}(D,J_t)$, 
and $\|K_{t}\alpha\|_s \le C \|\alpha\|_s$ for some $C>0$ which 
can be chosen independent of $t$ in any compact subset of $T$.
\item 
If the forms $\alpha_t\in H^{p,q}_{s}(D,J_t)$ depend continuously 
on $t\in T$, then $K_{t}\alpha_t\in H_{s-1}^{p,q-1}(D,J_t)$ also 
depend continuously on $t\in T$.
\item 
If $k\ge 0$, $s > k + n +1$, and the forms 
$\alpha_t\in H^{0,1}_{s}(D,J_t)$ depend continuously on $t\in T$, 
then the functions $K_t \alpha_t \in \Cscr^k(D)$ 
depend continuously on $t\in T$.
\item 
If $bD$ is $\Cscr^\infty$ smooth,  $\Jscr$ is of class $\Cscr^{0,\infty}$, 
and $\alpha_t\in \Dscr^{p,q}(\bar D,J_t)$ are smooth and continuous in $t$, 
then $K_{t}\alpha_t\in \Dscr^{p,q-1}(\bar D,J_t)$ are 
also smooth and continuous in $t\in T$.
\end{enumerate}
\end{theorem}

The Kohn solution $\phi=K_t\alpha$ of  
the equation $\dibar_{J_t} \phi=\alpha$, subject to 
$\dibar_{J_t} \alpha=0$, 
is given by $\phi=\vartheta_t N_t\alpha$, where $N_t$ is the 
$\dibar$--Neumann operator associated to $J_t$ and $\vartheta_t$ 
is the Hilbert space adjoint of $\dibar_{J_t}$ on $D$;
see \cite[Theorem 3.1.14]{FollandKohn1972}. 
The same result holds if $\bar D$ is a compact smooth manifold 
with boundary that is not necessarily embedded in an ambient manifold. 
Part (d) follows from (c) and the following Sobolev embedding theorem;
see Adams \cite[p.\ 97ff]{AdamsRA1975} or
Folland and Kohn \cite[Proposition A.1.2, p.\ 115]{FollandKohn1972} 
for $X=\R^N$ and note that the general case follows by using charts 
(see \cite[p.\ 122]{FollandKohn1972}).
Part (e) holds because the forms $K_{t}\alpha_t$ in (a)
are independent of the smoothness class
(see \cite[p.\ 55]{GreeneKrantz1982}). 

%
%
\begin{proposition}[Sobolev embedding theorem]
\label{prop:Sobolev}
Let $D$ be a relatively compact domain with $\Cscr^1$ boundary
in a smooth manifold $X$ of dimension $N$. Then, 
$H_s(D)\subset \Cscr^k(D)$ and 
$\|\cdotp\|_{\Cscr^k(D)} \le C \|\cdotp\|_{H_s(D)}$ for some $C>0$ 
if and only if $s > k + N/2$. If this holds then the weak 
derivatives of $u\in H_s(D)$ up to order $k$ are, after correction on a set 
of measure zero, classical derivatives. 
\end{proposition}

\begin{proof}[Proof of Theorem \ref{th:Hamilton}]
Choose a proper $J$-holomorphic embedding $f:X\hra \C^{2n+1}$
where $n=\dim X$
(see \cite{Bishop1961AJM} and \cite[Theorem 2.4.1]{Forstneric2017E}). 
By Docquier and Grauert \cite{DocquierGrauert1960} 
(see also \cite[Theorem 8, p. 257]{GunningRossi1965} 
or \cite[Theorem 3.3.3, p.\ 74]{Forstneric2017E}) 
there are an open neighbourhood $U\subset \C^{2n+1}$ of $f(X)$
and a holomorphic retraction $\tau:U\to f(X)$. 
Set $s = k + n +2$, so $2s+5=2k+2n+9=r$. 
(If $k=\infty$, we take $s=r=\infty$.) 
Recall that the family $\{J_t\}_{t\in T}$ 
is of class $\Cscr^{0,r}$ and $J_{t_0}$ is the restriction of $J$ to $\bar D$.
Note that $\alpha_t:=\dibar_{J_t}(f|_{\bar D})$ for $t\in T$ 
is a $\C^{2n+1}$-valued $(0,1)$-form with respect to $J_t$, 
of class $\Cscr^{r}(\bar D)$ and hence in $H^{0,1}_s(\bar D,J_t)$,
depending continuously on $t\in T$. 
By Theorem \ref{th:dibar} (d), for every $t\in T$ close to $t_0$ 
there is a unique solution 
$\phi_t$ of $\dibar_{J_t}\phi_t=\alpha_t$ and 
$\phi_t\perp\ker(\dibar_{J_t})$, with $\phi_t\in \Cscr^k(D)$
depending continuously on $t$. 
The map $f_t=f-\phi_t: D\to\C^{2n+1}$ is then $J_t$-holomorphic
and continuous $t$ as an element of the space $\Cscr^k(D)^{2n+1}$.
For $t=t_0$ we have $J_{t_0}=J$ and hence $\alpha_{t_0}=\dibar_J f=0$, 
$\phi_{t_0}=0$, and $f_{t_0} = f|_{D}$. It follows that 
for $t$ close enough to $t_0$ the map $f_t$ is so close to $f|_{D}$ 
in $\Cscr^k(D)^{2n+1}$ that its image belongs to $U$. 
For such $t$, the map
$
	F_t=f^{-1}\circ \tau\circ f_t: D\to X
$
is well-defined, $(J_t,J)$-holomorphic, 
it depends continuously on $t$ as an element of $\Cscr^k(D,X)$, 
and $F_{t_0}=\Id_{D}$. It follows that $F_t$ is $(J_t,J)$-biholomorphic 
on $D$ for $t$ close enough to $t_0$.
\end{proof}

%
%
%
%
\section{A wild family of complex structures on $\R^{4}$}
\label{sec:wild}

In this section, we construct a smooth family $\{J_t\}_{t\in \R}$ of integrable 
complex structures on $\R^{2n}$ for any $n>1$ with wild behaviour of
holomorphic hulls near $t=0$. It is built by using a 
Fatou--Bieberbach map $\C^2\to\C^2$ with non-Runge 
image, constructed by Wold \cite{Wold2008}.
This example motivates the definition
of a tame family of complex structures; see Definition \ref{def:tame}.

A compact set $K$ in a complex manifold $(X,J)$ is said to be 
{\em holomorphically convex} or {\em $J$-convex} if 
$K$ equals its holomorphically convex hull (also called $J$-convex hull),
defined by 
\[ 
	\wh K_J = 
	\big\{p\in X: |f(p)|\le \max_{x\in K} |f(x)|\ \ 
	\text{for all}\ f\in\Oscr_J(X)\big\}.
\] 
Here, $\Oscr_J(X)$ denotes the algebra of $J$-holomorphic functions on $X$.
When $J$ is the standard complex structure on $X=\C^n$ 
then $\wh K_J$ is the polynomial hull of $K$. 
See H\"ormander \cite{Hormander1990,Hormander1994} and 
Stout \cite{Stout2007} for further information on holomorphic convexity.

If $X$ is an open Riemann surface then a compact subset $K\subset X$ 
is holomorphically convex if and only if $X\setminus K$ has no relatively 
compact connected components. This is a topological condition 
independent of the choice of the complex structure.
This fact plays an important role in the proof of the 
Oka principle in \cite[Theorem 1.6]{Forstneric2024Runge}
for maps from families of complex structures on a smooth 
open surface to an Oka manifold.  
When attempting to obtain analogous results for 
families of integrable Stein structures $\{J_t\}_{t\in T}$
on a smooth open manifold $X$ of dimension $2n\ge 4$, 
one of the problems 
concerns the behaviour of $J_t$-convex hulls $\wh K_{\! J_t}$ 
of a compact set $K\subset X$ with respect to $t$. 
The following result shows that when $X=\R^{2n}$, $n>1$, 
the hulls can explode when $t\in T$ approaches a limit value $t_0\in T$.
As shown in Lemma \ref{lem:bounded} and Corollary \ref{cor:extension},
this phenomenon implies major restrictions on families of global 
$J_t$-holomorphic functions on $X$. 

%
%
\begin{theorem}\label{th:nontame}
Given a compact set $K\subset \R^{2n}$ $(n>1)$ with nonempty interior,
there is a family of integrable smooth complex structures 
$\{J_t\}_{t\in\R}$ on $\R^{2n}$, depending smoothly on $t\in \R$, 
such that $J_0$ is the standard complex structure on $\C^n$,
$(\R^{2n},J_t)$ is biholomorphic to $(\R^{2n},J_0)\cong\C^n$ 
for every $t\in\R$, and for any neighbourhood $U\subset \R$ of $0\in\R$ 
the set $\bigcup_{t\in U} \wh{K}_{J_t}\subset \R^{2n}$ 
is unbounded. 
\end{theorem}

\begin{proof} 
It suffices to consider the case when $K$ is the closed unit ball 
in $\R^4\cong \C^2$. 

Let $\C^*=\C\setminus\{0\}$.
By Wold \cite{Wold2008}, there is an injective 
holomorphic map $\Phi:\C^2\hra \C^2$ such that 
$\Phi(\C^2)\subset \C^*\times \C$ but the polynomial hull
$\widehat{\Phi(K)}$ of $\Phi(K)$ contains the origin 
$0\in \C^2$. In particular, $\widehat{\Phi(K)}\not\subset \Phi(\C^2)$ 
and hence $\Phi(\C^2)$ is not Runge in $\C^2$. 
We shall construct a smooth family of diffeomorphisms
$\Psi_t:\C^2 \to \Psi_t(\C^2) \subset \C^2$, $t\in\R$, such that 
for every $t\ne 0$ we have $\Psi_t(\C^2)=\C^2$ and 
$\Psi_t=\Phi$ holds on a neighbourhood of $t^{-1}K$,
while $\Psi_0=\Phi$ on $\C^2$. 

Assume for a moment that such a family $\Psi_t$ exists.
Let $J_t$ denote the complex structure on $\R^4\cong\C^2$ 
obtained by pulling back by $\Psi_t$ the standard complex structure
$\Jst$ on $\C^2$. In other words, the map 
$\Psi_t:\C^2\to \C^2$ is a biholomorphism from 
$(\C^2,J_t)$ onto $(\C^2,\Jst)$ when $t\ne 0$, and $\Psi_0=\Phi$ is a 
biholomorphism onto the non-Runge domain 
$\Phi(\C^2)\subset \C^2$ when $t=0$. Note that $J_t$
depends smoothly on $t$ since $\Psi_t$ does, and it agrees with $\Jst$ 
on a neighbourhood of $t^{-1}K\supset K$ since on this set we have that 
$\Psi_t=\Phi$, which is $\Jst$-holomorphic. 
For $t\ne 0$, the $J_t$-convex hull of $K$ equals
\begin{equation}\label{eq:hullJt}
	\wh{K}_{J_t} = \Psi_t^{-1} (\widehat{\Psi_t(K)}) = 
	\Psi_t^{-1} (\widehat{\Phi(K)})
\end{equation}
where the second equality follows from the fact that
$\Psi_t =\Phi$ on $t^{-1}K \supset K$. 
We claim that the set $\wh{K}_{J_t} \setminus t^{-1}K$ 
is nonempty for every $t\ne 0$. Indeed, if $\wh{K}_{J_t}\subset t^{-1}K$
then, since $\Phi=\Psi_t$ on $t^{-1}K$, it follows from \eqref{eq:hullJt} that 
$\Phi(\wh{K}_{J_t})=\Psi_t(\wh{K}_{J_t})=\widehat{\Phi(K)}$,
a contradiction to $\widehat{\Phi(K)}\not\subset \Phi(\C^2)$.
As $t\to 0$, the sets $t^{-1}K$ increase to $\C^2$,
and hence the hulls $\wh{K}_{J_t}$ are not contained in 
any bounded subset of $\C^2$ for $t$ in a neighbourhood of $0$. 

It remains to explain the construction of the family of diffeomorphisms 
$\Psi_t:\C^2 \to\C^2$ with the stated properties. It suffices to consider
the parameter values $t\in (0,1)$. 
Choose a smooth isotopy of injective holomorphic maps 
$\Phi_s:\C^2 \to \C^2$ for $s\in [0,1]$ such that $\Phi_0$ is the identity 
map on $\C^2$ and $\Phi_1=\Phi$. Explicitly, we can take $\Phi_0(z)=z$ and
\[
	\Phi_s(z) = s\, \Phi(0) + 
	s^{-1} A_s \big(\Phi(sz)-\Phi(0)\bigr), \quad s\in (0,1], 
\]
where $s\mapsto A_s\in GL_2(\C)$ is a smooth path with 
$A_0=\Phi'(0)^{-1}$ and $A_1=I$.
Note that $\{\Phi_s\}_{s\in [0,1]}$ is the flow of the holomorphic 
time-dependent vector field $V$ on $\C^2$  
defined on the open set
\[
	\Sigma=\bigl\{\bigl(s,\Phi_{s}(z)\bigr) : s\in [0,1],\ z\in \C^2 \bigr\} 
	\subset [0,1]\times \C^2
\]
(the trace of the isotopy $\{\Phi_s\}_{s\in [0,1]}$) by 
\[
	V(s,\Phi_s(z)) = \frac{\di}{\di u} \Big|_{u=s} \Phi_u(z).  
\]
For a fixed $t\in (0,1)$ consider the compact set 
\[ 
	\Sigma_t=\bigl\{\bigl(s,\Phi_{s}(z)\bigr) : s\in [0,1],\ z\in t^{-1}K \bigr\} 
	\subset [0,1]\times \C^2.
\] 
Pick a smooth function $\chi:(0,1) \times [0,1]\times \C^2\to [0,1]$ 
such that for every $t\in (0,1)$ the function 
$\chi(t,\cdotp,\cdotp):[0,1]\times \C^2\to [0,1]$ equals
$1$ on a neighbourhood of $\Sigma_t$ and has compact support.
For $(t,s)\in (0,1)\times [0,1]$ we define a vector field $W_{t,s}$ on $\C^2$ by 
\[
	W_{t,s}(z) =\chi(t,s,z)V(s,z), \quad z\in\C^2. 
\]
Note that $W_{t,s}$ is smooth in all variables, it agrees with 
$V(s,\cdotp)$ on a neighbourhood of $\Sigma_t$, and has compact 
support in $[0,1]\times \C^2$ for every fixed $t\in (0,1)$. 
It follows that the flow $\Psi_{t,s}$ of $W_{t,s}$ with 
respect to the variable $s\in[0,1]$, with $t\in (0,1)$ as a parameter, 
solving the initial value problem 
\[
	\frac{\di}{\di u}\Big|_{u=s} \Psi_{t,u}(z) = W_{t,s}(\Psi_{t,s}(z)),
	\quad \Psi_{t,0}(z)=z,  
\]
exists for all $s\in [0,1]$ and $z\in\C^2$, it agrees
with the flow of $V$ for $z\in t^{-1}K$ (which is $\Phi_s(z)$),
and is fixed near infinity in the $z$ variable since $W_{t,s}$ has
compact support. It follows that every map $\Psi_{t,s}:\C^2\to\C^2$
for $t\in (0,1),\ s\in [0,1]$ is a diffeomorphism onto $\C^2$. 
Setting $s=1$ gives a family of diffeomorphisms 
$\Psi_t=\Psi_{t,1}:\C^2\to\C^2$, $t\in (0,1)$, 
with the stated properties. Clearly the family extends smoothly
to all $t\in \R$.
\end{proof}

The following implies that one cannot do any serious analysis
for families of $J_t$-holomorphic functions
for $\Jscr=\{J_t\}_{t\in \R}$ in Theorem \ref{th:nontame}.
See also Corollary \ref{cor:extension}.

\begin{lemma}\label{lem:bounded}
(Notation as above.) If $f$ is a holomorphic function on 
$\Omega=\Phi(\C^2)\subset\C^2$ such that $f\circ \Phi \in\Oscr(\C^2)$
extends to a continuous family $f_t\in \Oscr_{J_t}(\C^2)$ for
$t$ near $0$, then $f$ is bounded on the set
$\widehat{\Phi(K)}\cap \Omega$, which is not relatively 
compact in $\Omega$.
\end{lemma}

\begin{proof}
From $\wh{K}_{J_t} =\Psi_t^{-1} (\widehat{\Phi(K)})$ (see \eqref{eq:hullJt}) 
we get $\wh{K}_{J_t} \cap t^{-1}K = \Psi_t^{-1} 
(\widehat{\Phi(K)}) \cap t^{-1}K $ for $t\ne0$. 
Since $\Psi_t=\Phi$ on $t^{-1}K$, it follows that 
\[
	\Psi_t(\wh{K}_{J_t} \cap t^{-1}K) = 
	\widehat{\Phi(K)} \cap \Phi(t^{-1}K).
\]
When $t\to 0$, the set on the right hand side increases
to $\widehat{\Phi(K)}\cap \Omega$. 
Choose a point $p\in \widehat{\Phi(K)}\cap \Omega$; hence
$p \in \widehat{\Phi(K)} \cap \Phi(t^{-1}K)$ for all small enough $t\ne 0$.
Note that $p_t:=\Psi_t^{-1}(p)\in \wh{K}_{J_t} \cap t^{-1}K$
converges to $p_0=\Phi^{-1}(p)$ as $t\to 0$.
Let $f\in \Oscr(\Omega)$. Suppose that there is a continuous 
family of holomorphic functions $f_t\in \Oscr_{J_t}(\C^2)$ for 
$t$ near $0$ such that $f_0=f\circ \Phi$.
Since $p_t\in \wh{K}_{J_t}$, we have 
$|f_t(p_t)|\le \max_K |f_t|$. Letting $t\to 0$ gives
$|f_0(p_0)|\le \max_K |f_0|$, which is equivalent to
$|f(p)| \le \max_{\Phi(K)} |f|$. This shows that 
$f$ is bounded on the set $\widehat{\Phi(K)}\cap \Omega$
as claimed.  
\end{proof}

\begin{remark}
The construction in the proof of Theorem \ref{th:nontame} works 
on any contractible Stein manifold $X$ which admits an injective holomorphic 
map $\Phi:X\to X$ such that, for some compact subset 
$K\subset X$ with nonempty interior, we have that 
$\widehat {\Phi(K)} \not\subset \Phi(X)$.
Besides $\C^n$, an example is any bounded convex domain 
$X$ in $\C^n$ for $n>1$. Indeed, assume that $X$ is such, and 
let $K\subset X$ be a compact set with nonempty interior.
By translation we may assume that $0\in \mathring K$.
Let $\Phi:\C^n\to\C^n$ be an injective holomorphic map
as in the proof of Theorem \ref{th:nontame}, satisfying 
$\wh{\Phi(K)}\not\subset\Phi(\C^n)$. Set $L=\Phi(K)$.
For any $s>0$ we then have $\wh{sL} = s\, \wh{L} \not\subset s \Phi(\C^n)$.
Replacing $\Phi$ by $s\Phi$ for a suitable $s>0$ 
we ensure that $\Phi(X)\subset X$. It follows 
that $\wh{\Phi(K)}\not\subset\Phi(X)$. 
However, we do not know whether the phenomenon in 
Theorem \ref{th:nontame} can occur on every Stein manifold $X$ 
with $\dim_\C X >1$.
\end{remark}

%
%
%
%
\section{Tame families of Stein structures}\label{sec:tame}

Assume that $T$ is a topological space, 
$X$ is a smooth open manifold of even dimension, $\pi:T\times X\to T$ 
is the projection $\pi(t,x)=t$, and $\mathscr{J}=\{J_t\}_{t\in T}$ 
a continuous family of integrable complex structures on $X$. 
We introduce a tameness condition on $\mathscr{J}$ 
which excludes the pathology in Theorem \ref{th:nontame}; 
see Definition \ref{def:tame}. If $T$ is locally compact and
Hausdorff and $J_t$ is Stein for every $t\in T$,
tameness is characterised in terms of 
properness over $T$ of the family of $J_t$-convex hulls of any
compact set in $X$; see Proposition \ref{prop:usc}. 
If the Stein structures $J_t$ are sufficiently regular, 
tameness is equivalent to 
local boundedness of the family of $J_t$-convex hulls
of any compact set; see Proposition \ref{prop:boundedhulls}. 
If $T$ is locally compact, paracompact and Hausdorff then tameness
is characterised by the existence of a continuous 
family of strongly $J_t$-plurisubharmonic exhaustion functions 
on $X$; see Theorem \ref{th:tame}.  
We conclude the section with examples and constructions 
of tame families of Stein structures.

%
%
\begin{definition} \label{def:tame}
A family $\mathscr{J}=\{J_t\}_{t\in T}$ of  
complex structures on $X$ is {\em tame} at a point 
$t_0\in T$ if for every compact set $K\subset X$ 
and open set $U\subset X$ containing $\wh K_{\! J_{t_0}}$ 
there is a neighbourhood $T_0\subset T$ of $t_0$ such that
$\wh K_{\! J_{t}} \subset U$ holds for all $t\in T_0$.
The family $\mathscr{J}$ is tame if it is tame at every
point $t_0\in T$.
\end{definition}

\begin{figure}[h]
  \centering
\includegraphics[scale=0.97]{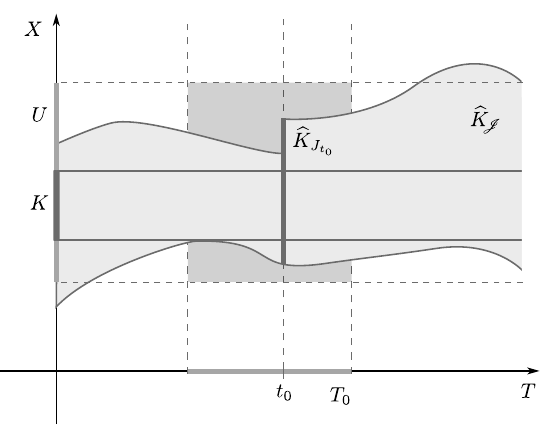}

\caption{An upper semicontinuous family of hulls $\wh{K}_{\! J_t}$.}
\end{figure}

Any family of complex structures on a smooth surface $X$ 
is tame since the hull of a compact set 
does not depend on the choice of a complex structure on $X$. 
The same holds if $X$ is compact and connected.
Theorem \ref{th:nontame} gives smooth nontame families 
of Stein structures on $\R^{2n}$ for any $n>1$.

Recall that a continuous map $S \to T$ of topological spaces is 
said to be proper if the preimage of any compact set
in $T$ is compact.

%
%
\begin{proposition}\label{prop:usc}
If $T$ is a locally compact Hausdorff space then 
the following conditions on a continuous 
family $\mathscr{J}=\{J_t\}_{t\in T}$ of Stein structures 
on $X$ are equivalent.
\begin{enumerate}[\rm (a)]
\item The family $\mathscr{J}$ is tame.
\item For every compact subset $K\subset X$, 
its $\mathscr{J}$-convex hull 
\begin{equation}\label{eq:whK}
	\wh K_{\!\!\mathscr{J}} =
	\bigcup_{t\in T} \, \{t\}\times \wh{K}_{\! J_t} \subset T\times X
\end{equation}
is such that the projection $\pi:\wh K_{\!\!\mathscr{J}}\to T$ is proper.
\end{enumerate}
\end{proposition}

\begin{proof}
Assume that $\Jscr$ is tame. It is easily seen
that $\wh K_{\!\!\mathscr{J}}$ is then closed in $T\times X$. 
Let $T'\subset T$ be compact. Given $t\in T'$, 
pick a neighbourhood $U_t \subset X$ of $\wh K_{\! J_{t}}$ 
with compact closure $\overline U_t$. 
Tameness gives a compact neighbourhood $T_t\subset T$ 
of $t$ such that the $\pi^{-1}(T_t)\cap \wh K_{\!\!\mathscr{J}}$
is a closed subset of $T_t\times \overline U_t$, hence compact. 
The compact set $T'$ is covered by finitely many sets 
$T_{t_j}$ obtained in this way, and it follows that 
$\pi^{-1}(T') \cap \wh K_{\!\!\mathscr{J}}$ is compact. 
This proves (a) $\Rightarrow$ (b). Conversely, assume that 
$\pi:\wh K_{\!\!\mathscr{J}}\to T$ is proper. 
Let $t_0\in T$ and $U \subset X$ be an open set
containing $\wh K_{\! J_{t_0}}$. Choose a compact neighbourhood
$T_0\subset T$ of $t_0$. Then, $\pi^{-1}(T_0)\cap \wh K_{\!\!\mathscr{J}}$
is compact, and hence closed in $T_0\times X$. 
If the condition in Definition \ref{def:tame} fails at $t_0$, there is a net 
$\{(t_j,x_j)\}_{j\in A}\subset \wh K_{\!\!\mathscr{J}}$ 
converging to $(t_0,x_0)$ with $x_0\in X\setminus U$. 
Since $\pi^{-1}(T_0)\cap \wh K_{\!\!\mathscr{J}}$
is closed, it contains $(t_0,x_0)$, a contradiction.
Hence, $\Jscr$ is tame.
\end{proof}

%
%
The nontame families $\Jscr$ of Stein structures in Theorem 
\ref{th:nontame} are such that the family of hulls $\wh K_{J_t}$
of some compact $K\subset X$ is not locally bounded at some $t_0\in T$,
and the hull $\wh K_{\!\!\mathscr{J}}$ \eqref{eq:whK} 
fails to be closed. We now show that this happens in any nontame 
family of sufficiently regular Stein structures.

%
%
\begin{proposition}\label{prop:boundedhulls}
Let $X$ be a smooth manifold of dimension $2n$ 
and $\Jscr=\{J_t\}_{t\in T}$ be a continuous 
family of Stein structures of class $\Cscr^{r}$ on $X$, where 
$r\geq 2n+9$ is an integer, such that for every compact set $K\subset X$
the family of hulls $\wh K_{J_t}$ is locally bounded.
Then the family $\Jscr$ is tame. 
\end{proposition}

\begin{proof}
Fix a compact set $K\subset X$ and a point 
$(t_0,x_0)\notin\widehat K_{\!\!\Jscr}$. We shall find a neighbourhood
of $(t_0,x_0)$ disjoint from $\widehat K_{\!\!\Jscr}$. 
By the assumption, there are a neighbourhood 
$T_0\subset T$ of $t_0$ and a domain $\Omega\Subset X$ 
such that $\wh K_{J_t} \subset \Omega$
holds for all $t\in T_0$, and $x_0\in \Omega$. 
We may choose $\Omega$ to be smoothly bounded  
and strongly $J_{t_0}$-pseudoconvex. 
Since $x_0\notin\wh K_{J_{t_0}}$, there is a 
$J_{t_0}$-holomorphic function $f$ on $X$ such that 
$|f(x_0)|>1+3\varepsilon$ and  $\max_{x\in K}|f(x)|\leq 1-\varepsilon$ 
for some $\varepsilon>0$. Let $U\subset \Omega$ 
be a compact neighbourhood of $x_0$ such that $|f(x)|> 1+3\varepsilon$ 
for all $x\in U$. Shrinking $T_0$ around $t_0$ if necessary,
the proof of Theorem \ref{th:Hamilton} applied with $k=0$ 
(see also Theorem \ref{th:dibar} (d)) 
gives a continuous family of functions 
$\{u_t\}_{t\in T_0}\subset \Cscr^0(\Omega)$ 
such that $\dibar_{J_t}u_t= \dibar_{J_t}f$ and $u_{t_0}=0$. 
Then $f_t=f-u_t$ is $J_t$-holomorphic on $\Omega$ and 
continuous in $t\in T_0$, with $f_{t_0}=f$. 
Hence, there is a neighbourhood $T_1\subset T_0$ 
of $t_0$ such that $\min_{x\in U} |f_t(x)|>1+2\varepsilon$ 
and $\max_{x\in K}|f_t(x)|\leq 1$ hold for all $t\in T_1$. 
We claim that $T_1\times U$ is disjoint from $\widehat K_{\!\!\Jscr}$. 
If not, choose a point 
$(t',x')\in  (T_1\times U) \cap \widehat K_{\!\!\Jscr}$, so 
$x'\in \wh K_{J_{t'}} \subset \Omega$. The Oka--Weil theorem gives a 
$J_{t'}$-holomorphic function $F$ on $X$ such that 
$|F-f_{t'}|<\varepsilon$ on $\wh K_{J_{t'}}$, which implies 
$|F(x')|>\max_{x\in K}|F(x)|$ in contradiction to $x'\in \wh K_{J_{t'}}$.
This proves the claim and shows that $\widehat K_{\!\!\Jscr}$ is closed.
Given an open neighbourhood $V\Subset \Omega$ of 
$\wh K_{J_{t_0}}$, the above argument shows that the compact set 
$\{t_0\} \times (\overline \Omega \setminus V)$ can be covered
by finitely many open sets 
$T_j \times U_j\subset (T\times X)\setminus \widehat K_{\!\!\Jscr}$ 
$(j=1,\ldots,m)$. Taking $T_0=\bigcap_{j=1}^m T_j$ it follows that 
$\pi^{-1}(T_0) \cap \widehat K_{\!\!\Jscr} \subset T_0\times V$,
which shows that $\Jscr$ is tame. 
\end{proof}

In the remainder of the section, 
we assume that the parameter $T$ is locally compact 
Hausdorff. A closed subset $K\subset T\times X$ is  
called {\em proper over $T$}, or simply {\em proper}, 
if the restricted projection $\pi|_K:K\to T$ is proper. 
The proof of Proposition \ref{prop:usc}
shows that $K$ is proper if and only if the fibres
$K_t=\{x\in X: (t,x)\in K\}$, $t\in T$, are compact 
and upper semicontinuous. 
Given a subset $K\subset T\times X$ with compact fibres $K_t$,
we define its $\mathscr{J}$-convex hull by 
\begin{equation}\label{eq:whK2}
	\wh K_{\!\!\mathscr{J}}
	=\big\{(t,x)\in T\times X: x\in (\wh{K_t})_{\! J_t}\big\}. 
\end{equation}
(In \eqref{eq:whK} we used the same notation for $K\subset X$ 
to mean $(\wh{T\!\times \!K})_{\!\!\mathscr{J}}$, but this should not cause 
any confusion.) A proper subset $K\subset T\times X$ 
is said to be $\mathscr{J}$-convex 
if $K= \wh K_{\!\!\mathscr{J}}$. 

%
%
\begin{lemma}\label{lem:tame}
If $K\subset T\times X$ is proper and 
$\mathscr{J}$ is a tame family of Stein structures on $X$
then the hull $\wh K_{\!\!\mathscr{J}}$ 
\eqref{eq:whK2} is also proper.
\end{lemma}

\begin{proof}
Fix $t_0\in T$ and an open set $U\subset X$ with 
$(\wh K_{t_0}\!)_{J_{t_0}} \subset U$. 
By \cite[Theorem 5.1.6]{Hormander1990} 
there is a strongly $J_{t_0}$-plurisubharmonic exhaustion function 
$\rho:X\to \R$ such that $\rho<0$ on $(\wh K_{t_0})_{\! J_{t_0}}$ and 
$\rho>0$ on $X\setminus U$. The compact set 
$L=\{\rho\le 0\}$ is then $J_{t_0}$-convex and satisfies
$K_{t_0} \subset \mathring L\subset L\subset U$.
Since $K$ is proper, there is a neighbourhood $T_0\subset T$ 
of $t_0$ such that $K_t\subset L$ for all $t_0\in T$. 
Since $\Jscr$ is tame, we have that
$(\wh K_t)_{J_t} \subset {\wh L}_{J_t} \subset U$ 
for all $t$ near $t_0$, so $\wh K_{\!\!\mathscr{J}}$ is proper.
\end{proof}

We have the following characterisation 
of tameness in terms of families of strongly $J_t$-plurisubharmonic exhaustion functions on $X$ for $t\in T$.

%
%
\begin{theorem}\label{th:tame}
Assume that $X$ is a smooth manifold, $T$ is a locally
compact Hausdorff space, and 
$\mathscr{J}=\{J_t\}_{t\in T}$ is a continuous family of integrable Stein 
structures on $X$ of local H\"older class $\Cscr^{0,(k,\alpha)}(T\times X)$
for some $k\in\N$ and $0<\alpha<1$. The following 
conditions are equivalent.
\begin{enumerate}[\rm (a)]
\item The family $\mathscr{J}$ is tame. 
\item For every $t_0\in T$ there are a neighbourhood $T_0\subset T$
of $t_0$ and a function $\rho:T_0\times X\to \R$ of class
$\Cscr^{0,k+1}$ such that $\rho(t,\cdotp)$ is a 
strongly $J_t$-plurisubharmonic exhaustion on $X$
for every $t\in T_0$. 
\end{enumerate}
If $T$ is paracompact then (b) is equivalent to the following:
\begin{enumerate}[\rm (c)]
\item There is a function $\rho:T\times X\to \R$ of class
$\Cscr^{0,k+1}$ such that $\rho_t=\rho(t,\cdotp)$ is a 
strongly $J_t$-plurisubharmonic exhaustion function on $X$
for every $t\in T$.  
\end{enumerate}
If $T$ is a $\Cscr^{l}$ manifold and $\Jscr$ is of local class 
$\Cscr^{l,(k,\alpha)}$, then $\rho$ can be chosen 
to be of class $\Cscr^{l,k+1}$. 
\end{theorem}

Note that the $\pm \imath$-eigenspaces of $J_t$ depend 
algebraically on the coefficients of $J_t$ in a given smooth
frame on $TX$, and hence the operators $\di_{J_t},\ \dibar_{J_t}$,
and $d^c_{J_t}$ are as regular in 
$t\in T$ as the family $\mathscr{J}=\{J_t\}_{t\in T}$.
In the operator $dd^c_{J_t}$, the coefficients 
of $d^c_{J_t}$ get differentiated, and hence 
$dd^c_{J_t}$ depends continuously on $t\in T$
if the family $\Jscr$ is of class $\Cscr^{0,1}$. 
This implies the following observation.

\begin{lemma}\label{lem:spsh}
Assume that the family of complex structures $\Jscr=\{J_t\}_{t\in T}$
on $X$ is of local class $\Cscr^{0,(1,\alpha)}$, $0<\alpha<1$. 
Let $\phi$ be a $\Cscr^2$ strongly $J_{t_0}$-plurisubharmonic function
on a domain $V \subset X$ for some $t_0\in T$. Given
an open relatively compact subset $U\Subset V$, there is a 
neighbourhood $T_0\subset T$ of $t_0$ such that 
$\phi$ is strongly $J_t$-plurisubharmonic on $U$ for every $t\in T_0$.
\end{lemma}

%
%
\begin{proof}[Proof of Theorem \ref{th:tame}] 
For simplicity of notation we assume that $k=1$ and $l=0$; 
the proof is the same in the general case. 

We first prove that (b)\ $\Rightarrow$\ (a); this implication
holds for arbitrary topological space $T$.
Fix $t_0\in T$, a compact set $K\subset X$, 
and an open relatively compact set $U\Subset X$ 
containing $\wh K_{\! J_{t_0}}$. Choose a neighbourhood
$T_0\subset T$ of $t_0$ and a function $\rho:T_0\times X\to \R$
satisfying condition (b).
By adding a constant to $\rho_{t_0}$ we can ensure that 
$\rho_{t_0}<-1$ on $K$. Since $\rho_{t_0}$ is
an exhaustion function on $X$, there is a relatively 
compact domain $V\Subset X$ containing $\overline U$
such that $\rho_{t_0}>1$ on $X\setminus V$. Choose a strongly 
$J_{t_0}$-plurisubharmonic function $\psi:X\to\R$ such that 
$\psi<0$ on $\wh K_{\! J_{t_0}}$ and $\psi>0$ on $X\setminus U$
(see \cite[Theorem 5.1.6]{Hormander1990}).
Replacing $\psi$ by $c\psi$ for a suitable $c>0$ 
we may assume that $-1<\psi<0$ on $K$,
$\psi>0$ on $X\setminus U$, and $\psi< \rho_{t_0}$ on $bV$. 
Since $\rho_t$ is continuous in $t$, we can shrink $T_0$ around $t_0$
to ensure that for every $t\in T_0$ we have $\rho_t <-1$ on $K$ and 
$\psi< \rho_{t}$ on $bV$. By Lemma \ref{lem:spsh} we can 
further shrink $T_0$ to ensure
that $\psi$ is strongly $J_t$-plurisubharmonic on $V$
for every $t\in T_0$. For $t\in T_0$ we 
define the function $\phi_t:X\to \R$ by 
\[
	\phi_t(x) = 
	\begin{cases} 
		\max\{\psi(x),\rho_{t}(x)\}, &  x\in V; \\
		\rho_t(x), & x\in X\setminus V.
	\end{cases}
\]
Note that $\phi_t$ is a piecewise $\Cscr^2$ strongly 
$J_t$-plurisubharmonic exhaustion function on $X$ satisfying 
\begin{equation}\label{eq:inequalities}
	\text{$\phi_t=\psi<0$ on $K$ and $\phi_t>0$ on $X\setminus U$.} 
\end{equation}
Since the holomorphic hull of $K$ equals its plurisubharmonic 
hull (see \cite[Theorems 4.3.4 and 5.2.10]{Hormander1990}), 
it follows from \eqref{eq:inequalities} that $\wh K_{\!J_t}\subset U$ 
for all $t\in T_0$. This shows that $\Jscr$ is tame. 

Next, we prove that (a)\ $\Rightarrow$\ (b). Since the statement
in (b) is local in $t$, we may assume that $T$ is compact.
Tameness of $\Jscr$ and compactness of $T$
imply that for every compact set $K\subset X$, 
the hull $\wh K_{\!\!\Jscr}$ \eqref{eq:whK} 
is also compact. Hence, we can find an exhaustion 
$	K^0\subset K^1\subset K^2\subset \cdots \subset 
	\bigcup_{i=0}^\infty K^i=X
$
by compact sets such that 
$\wh {K^i}_{\!\!\!\Jscr} \subset T\times \mathring K^{i+1}$
holds for every $i=0,1,2,\ldots$. 
Choose an increasing sequence $0<c_1<c_2<\cdots$ with 
$\lim_{i\to\infty}c_i=+\infty$. We proceed inductively. 

In the initial step, fix a neighbourhood $U^1\Subset X$ of $K^1$
and choose an open subset $\wt U^1\Subset X$
such that $\overline {U^1}\subset \wt U^1$. We shall find a 
function $\rho^1:T\times X\to \R_+$ of class $\Cscr^{0,2}$ 
such that $\rho^1_t$ is strongly plurisubharmonic on $U^1$
for all $t\in T$, $\rho^1$ 
has compact support contained in $T\times \wt U_1$, 
and $\rho^1>c_1$ on $T\times K^1$.
To do this, fix $t\in T$ and pick a strongly $J_t$-plurisubharmonic
function $\phi_t:X\to \R_+$ such that $\phi_t > c_1$ on $K^1$.
Lemma \ref{lem:spsh} gives 
a neighbourhood $T_t\subset T$ of $t$ such that
$\phi_t$ is strongly $J_s$-plurisubharmonic on $\wt U^1$ 
for every $s\in T_t$. By compactness of $T$ we obtain a finite
covering $T=\bigcup_{j=1}^m T_j$ and for each $j=1,\ldots,m$
a $\Cscr^2$ function $\phi_j:\wt U^1\to \R_+$ 
such that $\phi_j>c_1$ on $K^1$ and $\phi_j$ is strongly 
$J_t$-plurisubharmonic for every $t\in T_j$. Let $\{\chi_j\}_{j=1}^m$
be a continuous partition of unity on $T$ with $\supp \chi_j\subset T_j$.
Also, let $\xi:X\to [0,1]$ be a smooth function with compact support
contained in $\wt U^1$ which equals $1$ on $U^1$.
The function $\rho^1(t,x)=\xi(x) \sum_{j=1}^m \chi_j(t) \phi_j(x)$
is then fibrewise strongly plurisubharmonic on $T\times U_1$ and has
compact support contained in $T\times \wt U^1$.

In the second step, we pick a neighbourhood $U^2\Subset X$ of 
$K^2$ and find a function $\rho^2:T\times X \to\R_+$ of class
$\Cscr^{0,2}$ with compact support such that $\rho^2=0$ on 
$\wh {K^0}_{\!\!\!\Jscr}$, $\rho^1+\rho^2$ is fibrewise 
strongly plurisubharmonic on $T\times U^2$,
and $\rho^1+\rho^2>c_2$ on $T\times (K^2\setminus K^1)$.
(We also have $\rho^1+\rho^2>c_1$ on $T\times K^1$.)
To do this, fix $t\in T$ and apply \cite[Theorem 5.1.6]{Hormander1990}
to find a smooth $J_t$-plurisubharmonic 
function $\phi_t:X\to \R_+$ which vanishes on a neighbourhood
of $\wh {K^0}_{\! J_t}$, it is positive strongly $J_t$-plurisubharmonic 
on $X \setminus K^1$ (recall that $\wh {K^0}_{\! J_t}$ is contained
in the interior of $K^1$), and $\rho^1_t+\phi_t$
is strongly $J_t$-plurisubharmonic on $X$ and 
satisfies $\rho^1_t+\phi_t>c_2$ on $K^2\setminus K^1$.
By tameness of $\Jscr$ and Lemma \ref{lem:spsh}
there is a neighbourhood $T_t\subset T$ of $t$ such that 
the function $\rho^1_s+\phi_t:U^2\to\R_+$ 
satisfies the same conditions for all $s\in T_t$,
and $\phi_t$ vanishes on a neighbourhood
of $\wh {K^0}_{\! J_s}$ for all $s\in T_t$. 
As in the first step, this gives a finite open covering 
$T=\bigcup_{j=1}^m T_j$, functions $\phi_j:X\to\R_+$
$(j=1,\ldots,m)$, a partition of unity $\{\chi_j\}_{j=1}^m$ on $T$ 
with $\supp \chi_j\subset T_j$,
and a smooth cut-off function $\xi:X\to[0,1]$ such that the function 
$\rho^2(t,x)=\xi(x) \sum_{j=1}^m \chi_j(t) \phi_j(x)$ enjoys 
the stated properties.

This process can be continued inductively to yield a sequence 
of nonnegative functions $\rho^1,\rho^2,\ldots$ 
of class $\Cscr^{0,2}(T\times X)$ with compact supports 
such that their partial sums $\tilde \rho^i=\rho^1+\cdots+\rho^i$
are of class $\Cscr^{0,2}$ and satisfy the following conditions
for every $i=1,2,\ldots$:
\begin{enumerate}[\rm (i)]
\item $\tilde \rho^i$ is fibrewise 
strongly plurisubharmonic on a neighbourhood
of $T\times K^i$ and has compact support.
\item $\tilde \rho^i>c_1$ on $T\times K^1$ and 
$\tilde \rho^i>c_j$ on $T\times (K^j\setminus K^{j-1})$
for $j=2,\ldots,i$.
\item $\tilde \rho^{i+1}=\tilde\rho^i$ on $\wh {K^i}_{\!\!\!\Jscr}$. 
\end{enumerate}
Condition (iii) implies that the sequence is stationary on 
any compact subset of $T\times X$. It follows that 
$\rho=\sum_{i=1}^\infty \rho^i : T\times X\to \R_+$ is a 
fibrewise strongly plurisubharmonic function of class 
$\Cscr^{0,2}$ satisfying
$\rho>c_i$ on $T\times (K^i\setminus K^{i-1})$ for every 
$i=1,2,\ldots$. In particular, $\rho_t=\rho(t,\cdotp)$ 
is an exhaustion function on $X$ for every $t\in T$. 
It is easy to ensure that the Levi form of $\rho_t$ 
with respect to $J_t$ grows as fast as desired uniformly in $t\in T$. 
This proves the implication (a) $\Rightarrow$ (b).

Assume now that $T$ is also paracompact.
If (b) holds, we obtain a locally finite open cover $\Vcal=\{V_i\}_i$ of $T$
and for every $i$ a fibrewise strongly plurisubharmonic
exhaustion function $\rho_i:T_i\times X\to\R$.  
Pick a partition of unity $\{\chi_i\}_i$ on $T$ 
subordinate to $\Vcal$. Then, the function 
$\rho=\sum_i \chi_i\rho_i :T\times X\to\R$ 
satisfies condition (c). The implication (c) $\Rightarrow$ (b)
is a tautology. 

If $T$ is a $\Cscr^{l}$ manifold, the same proof gives a 
function $\rho$ of class $\Cscr^{l,k+1}(T\times X)$.
\end{proof}

The proof of Theorem \ref{th:tame} gives the following analogue
of the classical result \cite[Theorem 5.1.6]{Hormander1990}
for a tame family of Stein structures. We leave the details 
to the reader.

\begin{theorem}\label{th:tame2}
Assume that $X$, $T$ and $\Jscr$ are as in Theorem \ref{th:tame}.
Given a proper $\Jscr$-convex subset 
$K=\wh K_{\!\!\Jscr}\subset T\times X$
and an open set $U\subset T\times X$ containing $K$, there
is a function $\rho:T\times X\to\R$ as in Theorem \ref{th:tame} (b)
such that $\rho<0$ on $K$ and $\rho>0$ on $(T\times X) \setminus U$.
Conversely, if $\rho$ is a function as in Theorem \ref{th:tame} (b)
then for every $c\in \R$ the sublevel set $\{\rho\le c\}\subset T\times X$
is proper and $\Jscr$-convex.  
\end{theorem}

%
%
We give a few examples and constructions of 
tame families of Stein structures. 
The first observation is that nontameness can only appear
due to the behaviour of the structures near infinity.

\begin{proposition}\label{prop:locallyconstant}
Let $X$ and $\Jscr=\{J_t\}_{t\in T}$ be as in Theorem \ref{th:tame},
where $T$ is an arbitrary topological space.
If for every $t_0\in T$ there are a compact set $K\subset X$ and
a neighbourhood $T_0\subset T$ of $t_0$ such that $J_t=J_{t_0}$
holds on $X\setminus K$ for all $t\in T_0$, 
then the family $\Jscr$ is tame.
\end{proposition}

\begin{proof}
Let $t_0\in T_0\subset T$ and $K\subset X$ be as in the assumptions.
Pick a strongly $J_{t_0}$-plurisubharmonic exhaustion function
$\rho:X\to \R_+$. Lemma \ref{lem:spsh} shows that, 
up to shrinking $T_0$ around $t_0$, $\rho$ is strongly 
$J_{t}$-plurisubharmonic on $K$ for every $t\in T_0$.
The same is true on $X\setminus K$ since $J_t=J_{t_0}$ there. 
Hence, the implication (b)\ $\Rightarrow$ (a)
in Theorem \ref{th:tame} shows that the family $\{J_t\}_{t\in T_0}$ is tame. Since tameness is a local condition in the parameter $t$, $\Jscr$ is tame.
\end{proof}

%
%
\begin{proposition}\label{prop:Runge}
If $(X,J)$ is a Stein manifold and $\Phi_t:X\to \Phi_t(X)\subset X$ is 
a continuous family of diffeomorphisms onto Stein Runge domains in $X$,
then the family of Stein structures $J_t=\Phi_t^*J$ on $X$ is tame.
\end{proposition}

\begin{proof}
Let $K\subset X$ be a compact set. Set $\Omega_t=\Phi_t(X)$
and $K_t=\Phi_t(K)$. Denote by $J^t_0$
the restriction of $J$ to $T\Omega_t$.
Since $\Phi_t:(X,J_t)\to (\Omega_t,J^t_0)$ is a biholomorphism,
we have $\wh K_{J_t}=\Phi_t^{-1}((\wh{K_t})_{J^t_0})$. 
Since $\Omega_t$ is Runge in $X$, 
$(\wh{K_t})_{J^t_0}$ equals $(\wh{K_t})_{J}$,
the hull of $K_t$ in $(X,J)$.
Since the family $K_t$ is continuous in $t$, 
the family of hulls $(\wh{K_t})_{J}$ is upper semicontinuous 
in $t$, so the same is true for $\wh K_{J_t}$.
\end{proof}

Theorem \ref{th:nontame} shows that Proposition \ref{prop:Runge}
fails in general if the domains $\Phi_t(X)$ are not Runge in $X$.
In that example, $X=\C^n$ with $n>1$, $t\in \R$, $\Phi_t(\C^n)=\C^n$
for $t\ne 0$, while $\Phi_0(\C^n)$ is not Runge in $\C^n$.
It is easy to find an example of a tame family of Stein structures 
$\{J_t\}_{t\in \R}$ on $\R^{2n}$ such that $(\R^{2n},J_0)$ 
equals $\C^n$ while $(\R^{2n},J_t)$ for $t\ne 0$ is biholomorphic 
to the unit ball in $\C^n$.

\begin{example}
Assume that $(Y,J_Y)$ is a Stein manifold, $X$ is a smooth manifold,
and $F:T\times X\to Y$ is a map of class $\Cscr^{0,\infty}$ such that for 
every $t\in T$, the map $F_t=F(t,\cdotp):X \to Y$ is a proper 
immersion whose image $F_t(X)$ is an immersed complex
submanifold of $Y$. Let $J_t$ denote the unique complex 
structure on $X$ such that the map $F_t$ is $(J_t,J_Y)$-holomorphic.
Since $F_t$ is proper, $J_t$ is Stein. 
Choosing a strongly plurisubharmonic exhaustion function 
$\rho:Y\to\R_+$, the function $\rho\circ F:T\times X\to\R_+$ satisfies
condition (c) in Theorem \ref{th:tame}, so $\Jscr=\{J_t\}_{t\in T}$ is tame.
\end{example}

%
%
%
%
\section{The Oka principle for tame families of Stein structures}
\label{sec:Oka}

In this section, we state and prove the main result of the paper,
Theorem \ref{th:Oka}. It gives a parametric Oka principle with approximation
for maps from tame families of Stein structures to any Oka manifold.
Except for the regularity assumptions and statements, 
this result extends the special case concerning families of open Riemann
surfaces in \cite[Theorem 1.6]{Forstneric2024Runge}. 

Let $T$ be a locally compact Hausdorff space, 
$X$ be a smooth open manifold, and let $\pi:T\times X\to T$ 
denote the projection. Assume that $\mathscr{J}=\{J_t\}_{t\in T}$
is a tame family of integrable Stein structures on $X$
(see Definition \ref{def:tame}). 
Recall that a closed subset $K\subset T\times X$ is 
called {\em proper over $T$} (or simply {\em proper}) 
if the restricted projection $\pi|_K:K\to T$ is proper,
and is called $\mathscr{J}$-convex if $K= \wh K_{\!\!\mathscr{J}}$
(see \eqref{eq:whK2}). 
A continuous map $f$ on an open $U\subset T\times X$ 
is said to be {\em $\mathscr{J}$-holomorphic} if the map 
$f_t = f(t, \cdot)$ is $J_t$-holomorphic on $U_t = \{x \in X : (t,x) \in U \}$ 
for every $t \in T$.  
A topological space is said to be {\em $\sigma$-compact} if it is 
the union of countably many compact subspaces.
Every locally compact and $\sigma$-compact Hausdorff
space is paracompact \cite{Michael1957PAMS}. 
A topological space $T$ is a {\em Euclidean neighbourhood retract} (ENR) 
if it admits a topological embedding $\iota: T \hra \R^N$ for some $N$ 
whose image $\iota(T)\subset \R^N$ is a neighbourhood retract,
and is a {\em local ENR} if every point of $T$ 
has an ENR neighbourhood. (See 
\cite[Definition 1.5]{Forstneric2024Runge} and the references therein.)
In particular, every finite CW complex is an ENR,
and every countable locally compact CW-complex of finite dimension
is an ENR.

%
%
\begin{theorem}[The Oka principle for tame families of Stein structures] \label{th:Oka}
Assume the following:
\begin{enumerate}[\rm (a)]
\item $T$ is a $\sigma$-compact Hausdorff local ENR. 
In particular, $T$ may be a finite CW complex or a 
countable locally compact CW-complex of finite dimension.
\item $X$ is a smooth open manifold of real dimension $2n$. 
\item $r\ge 2n+11$ is an integer, or $r=+\infty$.  
\item $\Jscr=\{J_t\}_{t\in T}$ is a tame family of Stein structures 
of class $\Cscr^{0,r}$ on $X$ (see Definition \ref{def:tame}).
\item $K\subset T\times X$ is a proper (over $T$) $\mathscr{J}$-convex subset.
\item $Y$ is an Oka manifold with a distance function $\dist_Y$ 
inducing the manifold topology.
\item $f:T\times X \to Y$ is a continuous map, and there are an open subset 
$U\subset T\times X$ containing $K$ and a closed subset $Q\subset T$ such that $f$ is $\mathscr{J}$-holomorphic on $U\cup (Q\times X)$.
\end{enumerate}
Given a continuous function $\epsilon:T\to (0,+\infty)$,  
there exist a neighbourhood $U'\subset U$ of $K$ and a homotopy 
$f_s:T\times X\to Y$ $(s\in I=[0,1])$ satisfying the following conditions.
\begin{enumerate}[\rm (i)]
\item $f_0=f$.
\item The map $f_{s}$ is $\mathscr{J}$-holomorphic on $U'$
for every $s\in I$.
\item $\sup_{x\in K_t}\dist_Y(f_s(t,x),f(t,x))<\epsilon(t)$
for every $t\in T$ and $s\in I$.
\item The map $F=f_1$ is $\mathscr{J}$-holomorphic on $T\times X$. 
\item
The homotopy $f_s(t,\cdotp)$ $(s\in I)$ is fixed for every $t\in Q$, 
so $F=f$ on $Q\times X$.
\end{enumerate}
\end{theorem}

\begin{remark}\label{rem:Hamilton}
The choice of the integer $r$ in condition (c) is dictated by 
Theorem \ref{th:Hamilton}. If $k\ge 1$ and $r\ge 2k+2n+9$ are integers 
or $k=r=+\infty$, it follows from Theorem \ref{th:Hamilton} that every
continuous $\Jscr$-holomorphic map $f:U\to Y$ on an open subset 
$U\subset T\times X$ is of class $\Cscr^{0,k}$. 
Approximation in the fine $\Cscr^{0,0}$ topology (see (iii) in the theorem)
can then be upgraded to approximation in the fine $\Cscr^{0,k}$ topology;
see the last paragraph in \cite[Theorem 1.6]{Forstneric2024Runge}.
We shall not formally state or prove this generalisation since 
it follows from standard arguments.
\end{remark}

We first explain the special case of 
Theorem \ref{th:Oka} with $Y=\C$. 
In the following version of the Oka--Weil theorem for 
tame families of Stein structures, the parameter space $T$ can  
be more general than in Theorem \ref{th:Oka}. 
The special case when $X$ is an open surface is 
given by \cite[Theorem 1.1]{Forstneric2024Runge}.  

%
%
%
\begin{theorem}[The Oka--Weil theorem for tame families of Stein structures]
\label{th:OkaWeil}
	Assume that $X$ is a smooth manifold of dimension $2n$,
	$T$ is a locally compact and paracompact Hausdorff space, 
	$k\ge 1$ and $r\ge 2k+2n+9$ are integers or $k=r=+\infty$, 
	$\!\mathscr{J}=\{J_t\}_{t \in T}$ is a tame family of Stein structures 
	of class $\Cscr^{0,r}$ on $X$, 
	$K\subset T\times X$ is a proper over $T$ and
	$\mathscr{J}$-convex subset, 
	$U \subset T \times X$ is an open set containing $K$, and
	$f:U\to\C$ is a $\mathscr{J}$-holomorphic function. 
	Then, $f\in \Cscr^{0,k}(U)$ and it 
	can be approximated in the fine $\Cscr^{0,k}$ topology on $K$
	by $\mathscr{J}$-holomorphic functions $F: T\times X\to \C$. 
	If in addition $Q$ is a closed subset of $T$ and 
	$f$ is also $\Jscr$-holomorphic on $Q\times X$,   
	then $F$ can be chosen such that $F=f$ on $Q\times X$.
\end{theorem}

Theorem \ref{th:OkaWeil} has the following corollary which 
shows that tameness of $\Jscr$ is implied by, and hence
equivalent to the one-fibre extension property for $\Jscr$-holomorphic
functions. 

%
%
%
%
\begin{corollary}\label{cor:extension}
Assume that $T$ is a locally compact and paracompact Hausdorff space, 
$X$ is a smooth manifold, and $\mathscr{J}=\{J_t\}_{t \in T}$ 
is a continuous family of smooth Stein structures on $X$.
\begin{enumerate}[\rm (a)]
\item If $\Jscr$ is tame then every $\Jscr$-holomorphic function 
on $Q\times X$, where $Q$ is a closed subset of $T$, extends to a 
$\Jscr$-holomorphic function on $T\times X$. 
\item Conversely, if for every $f\in \Oscr_{J_{t_0}}(X)$ $(t_0\in T)$ 
there are a neighbourhood $T_0\subset T$ of $t_0$ and a 
$\Jscr$-holomorphic function $F:T_0\times X\to \C$
such that $F(t_0,\cdotp)=f$, then the family $\Jscr$ is tame.
\end{enumerate}
\end{corollary}

\begin{proof}[Proof of Corollary \ref{cor:extension}]
Note that (a) is a part of Theorem \ref{th:OkaWeil}.
We prove (b) by contradiction. 
For simplicity we assume that $T$ is first countable; in the general
case the same argument works with sequences replaced by nets.
Assume that $\Jscr$ is not tame.
Then there are a point $t_0\in T$, a compact 
$J_{t_0}$-convex set $K\subset X$, a neighbourhood 
$U\Subset X$ of $K$, and a sequence $t_j\in T$ with 
$\lim_{j\to\infty} t_j=t_0$ such that the hull $\wh K_{J_{t_j}}$ is not
contained in $U$ for any $j$. Since $K\subset U$, 
it follows that $\wh K_{J_{t_j}} \cap bU\ne\varnothing$.
Pick a point $x_j\in \wh K_{J_{t_j}} \cap bU$ for every $j$. 
Since $bU$ is compact, passing to a subsequence we may assume 
that $x_j$ converges to a point $x_0\in bU$ as $j\to\infty$. 
Assume that  $T_0\subset T$  
is a neighbourhood of $t_0$ and $F:T_0\times X\to \C$ 
is a $\Jscr$-holomorphic function. Let $f=F(t_0,\cdotp)$.
For every sufficiently big $j$ we have $t_j\in T_0$ and
hence $|F(t_j,x_j)|\le \max_{x\in K} |F(t_j,x)|$. Taking the limit
as $j\to\infty$ gives $|f(x_0)|\le \max_{x\in K} |f(x)|$.
Since $x_0 \not\in \wh K_{J_{t_0}}$, there exists a
function $f\in \Oscr_{J_{t_0}}(X)$ violating the above inequality,
and hence such $f$ does not admit a $\Jscr$-holomorphic
extension to $T_0\times X$ for any neighbourhood $T_0$ of $t_0$.  
\end{proof}

%
%

For a tame family $\Jscr$, Corollary \ref{cor:extension} also implies 
the following characterisation of $\Jscr$-convex sets by 
$\Jscr$-holomorphic functions. 

\begin{corollary}\label{cor:Jconvex}
Let $T$, $X$, and $\mathscr{J}=\{J_t\}_{t \in T}$ be 
as in Corollary \ref{cor:extension}, and assume that $\Jscr$ is tame.
Then a proper subset $K\subset Z=T\times X$ is $\Jscr$-convex if 
and only if for every point $z_0=(t_0,x_0)\in Z\setminus K$ there exists
a $\Jscr$-holomorphic function $f:Z\to\C$ such that 
$|f(z_0)|>\sup_{z\in K} |f(z)|$.
\end{corollary}

\begin{proof}
If $K$ is not $\Jscr$-convex then one of its fibres
$K_{t_0}$ is not $J_{t_0}$-convex. 
If $x_0\in (\wh{K_{t_0}})_{J_{t_0}}\setminus K_{t_0}$ 
then every $\Jscr$-holomorphic function $f$ satisfies 
$|f(t_0,x_0)| \le \sup_{z\in K} |f(z)|$. Assume now that $K$ is $\Jscr$-convex
and let $z_0=(t_0,x_0)\in Z\setminus K$. Then, 
$x_0\not \in K_{t_0}=(\wh{K_{t_0}})_{J_{t_0}}$, so there exists 
$f_{t_0}\in\Oscr_{J_{t_0}}(X)$ with 
$|f_{t_0}(x_0)|>\max_{x\in K_{t_0}}|f_{t_0}(x)|$.
By Corollary \ref{cor:extension} (a) there is a $\Jscr$-holomorphic
function $F:Z\to \C$ with $F(t_0,\cdotp)=f_{t_0}$.
Since $K$ is proper, its fibres $K_t$ are upper
semicontinuous, so there is a neighbourhood $T_0\subset T$
of $t_0$ such that $|F(t_0,x_0)|>\max_{x\in K_t}|F(t,x)|$ 
for all $t\in T_0$. If $\chi:T\to [0,1]$ is a continuous function with 
$\chi(t_0)=1$ and $\supp \chi\subset T_0$ then 
the function $f:Z\to \C$ given by $f(t,x)=\chi(t) F(t,x)$ satisfies
the conclusion of the corollary.
\end{proof}

%
%
\begin{proof}[Proof of Theorem \ref{th:OkaWeil}]
%
%
We first consider the case when $T$ is compact and $Q=\varnothing$. 
Since the set $K\subset T\times X$ is proper over the compact set $T$, 
$K$ is compact as well. Choose a compact set $L\subset X$ 
such that $K\subset T\times L$ and set 
$L'=\wh {T\times L}_{\!\!\Jscr}$ \eqref{eq:whK2}. 
It suffices to show that the function $f$ in the theorem can be 
approximated as closely as desired in $\Cscr^{0,k}(K)$ 
by $\Jscr$-holomorphic functions $F:U'\to \C$ on an open 
neighbourhood $U'\subset T\times X$ of $L'$. If this holds  then  
the conclusion follows by an
induction with respect to an exhaustion of $T\times X$
by an increasing sequence of compact $\Jscr$-convex sets.

Consider the problem for $t\in T$ 
near a fixed $t_0 \in T$. We denote by $K_t\subset X$ the fibre 
of $K$ over $t\in T$, and likewise for the other sets. 
For a subset $T_0\subset T$ we also write $K_{T_0}=K\cap (T_0\cap X)$
and $U_{T_0} = U\cap (T_0\times X)$.
Choose a relatively compact strongly $J_{t_0}$-pseudoconvex 
domain $\Omega\Subset X$ with smooth boundary 
such that $L'_{t_0}=\wh L_{J_{t_0}} \subset \Omega$.
Theorem \ref{th:Hamilton} furnishes a neighbourhood $T_0\subset T$
of $t_0$ and a map $\Phi: T_0\times \Omega\to T_0\times X$ 
of class $\Cscr^{0,k}$ such that $\Phi(t,x)=(t,\Phi_t(x))$ and
\begin{equation}\label{eq:Phit}
	\Phi_t : \Omega \to \Phi_t(\Omega) \subset X \ \ 
	\text{is a $(J_t,J_{t_0})$-biholomorphism for every $t\in T_0$},
\end{equation} 
with $\Phi_{t_0} = \operatorname{Id}_\Omega$. 
Choose $J_{t_0}$-Stein domains $V,V'$ in $X$ such that 
$L'_{t_0}\subset V \Subset V' \Subset  \Omega$. 
Shrinking $U$ around $K$ and $T_0$ around $t_0$ if necessary, 
the following inclusions hold for every $t\in T_0$ (see Fig.\ \ref{fig:PhitKt}):
\begin{equation}\label{eq:inclusions}
	U_t\subset\Omega,\quad 
	\Phi_t(K_t) 
	\subset V \subset V' \subset \Phi_t(\Omega),\quad 
	L'_t\subset V \subset  \Phi_t^{-1}(V').
\end{equation}

\begin{figure}[h] 
	\centering
	\includegraphics[scale=0.95]{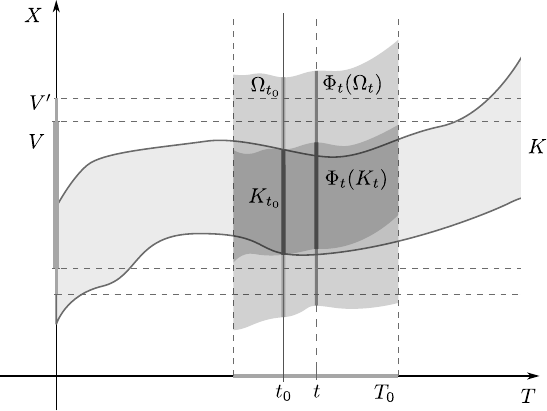}
  	\caption{The shaded areas depict the sets 
	$\Phi_t(K_t)\subset \Phi_t(\Omega)$ for $t\in T_0$.
	\label{fig:PhitKt}}
\end{figure}

Let $f:U\to\C$ be as in the theorem, so $f_t=f(t,\cdotp)$ is $J_t$-holomorphic 
on $U_t\supset K_t$ for every $t\in T$. 
The function $f\circ \Phi^{-1}:\Phi(U_{T_0}) \to \C$
is then continuous and fibrewise $J_{t_0}$-holomorphic, 
hence of class $\Cscr^{0,\infty}$. Since $K_t$ is $J_t$-convex in $\Omega$ 
and the map $\Phi_t$ \eqref{eq:Phit} is $(J_t,J_{t_0})$-biholomorphic, 
$\Phi_t(K_t)$ is $J_{t_0}$-convex in $\Phi_t(\Omega)$ 
(and hence in $V'\subset \Phi_t(\Omega)$, see \eqref{eq:inclusions}) 
for every $t\in T_0$. Hence, the set $K':=\Phi(K_{T_0})\subset T_0\times X$ 
is proper over $T_0$ and its fibres are $J_{t_0}$-convex in the Stein manifold
$(V',J_{t_0})$. By \cite[Lemma 5.3]{Forstneric2024Runge} 
there is a function $F': T_0\times V' \to \C$ of class $\Cscr^{0,\infty}$ 
which is fibrewise $J_{t_0}$-holomorphic and 
approximates $f\circ \Phi^{-1}$ as closely as desired in 
$\Cscr^{0,k}(K')$. Since $\Phi_t(V) \subset V'$ by the third inclusion in \eqref{eq:inclusions}, the function $F:=F'\circ \Phi:T_0\times V\to\C$
is well-defined, of class $\Cscr^{0,k}$, and it approximates 
$f$ in $\Cscr^{0,k}(K_{T_0})$. 
By the fifth inclusion in \eqref{eq:inclusions} we have that 
$L'\cap (T_0\times X) \subset T_0\times V$.
	
This gives a finite open cover $\{T_j\}_j$ of $T$ and 
open sets $V_j\subset X$ such that $\{T_j\times V_j\}_j$
is a cover of $L'$, and $\Jscr$-holomorphic functions 
$F_j:T_j\times V_j\to\C$ approximating $f$ in $\Cscr^{0,k}(K_{T_j})$
for every $j$. 
Choose a partition of unity $1 = \sum_j \chi_j$ on $T$ 
with $\operatorname{supp} \chi_j \subset T_j$ for every $j$. 
The function $F(t, x) = \sum_j \chi_j(t)F_j(t, x)$ 
is then well-defined and $\Jscr$-holomorphic 
on a neighbourhood of $L'$ in $T\times X$ and it 
approximates $f$ in $\Cscr^{0,k}(K)$. 
To conclude the proof, it remains to apply an induction with respect
to an exhaustion of $T\times X$ by an increasing family of 
$\Jscr$-convex sets. 

Suppose now that $T$ is compact and $Q\subset T$ is nonempty.
Let $K\subset L' \subset T\times X$ be as above. Choose a 
strongly pseudoconvex domain $\Omega \Subset X$ such that 
$L'\cap (Q\times X) \subset Q\times \Omega$.   
We claim that there is a neighbourhood $T'\subset T$ of $Q$ 
and a $\Jscr$-holomorphic function $f':T'\times \Omega\to \C$
which agrees with $f$ on $Q\times \Omega$. If the complex 
structure $J_t$ is independent of $t\in T'$, this follows 
from the parametric Oka--Weil theorem 
\cite[Theorem 2.8.4]{Forstneric2017E}. In the case at hand, 
we choose a pair of smoothly bounded strongly pseudoconvex domain 
$\Omega_1 \Subset \Omega_2\Subset X$ such that 
$\overline \Omega \subset \Omega_1$, and we cover $Q$ by finitely
many open sets $T_1,\ldots,T_m\subset T$ with points $t_j\in T_j$ 
such that Theorem \ref{th:Hamilton} applies on 
$\overline T_j\times \Omega_2$ for every $j=1,\ldots,m$. This gives  
maps $\Phi_j: \overline T_j\times \Omega_2 \to \overline T_j\times X$ 
of class $\Cscr^{0,k}$ and of the form \eqref{eq:Phit} such that 
$\Phi_{j,t}: \Omega_2 \to \Phi_{j,t}(\Omega_2)$
is a $(J_t,J_{t_j})$-biholomorphism for every $t\in \overline T_j$.
Choosing the sets $T_j$ small enough we may assume that 
the following inclusions hold for $j=1,\ldots,m$:
\begin{equation}\label{eq:inclusions3}
	\overline T_j \times \overline \Omega \subset
	\Phi_j^{-1}(\overline T_j\times \Omega_1), 
	\quad 
	\overline T_j \times \overline \Omega_1 \subset 
	\Phi_j(\overline T_j \times \Omega_2).  
\end{equation}
We apply \cite[Theorem 2.8.4]{Forstneric2017E} to each
function $f\circ \Phi_j^{-1}: (\overline T_j\cap Q) \times \Omega_1\to\C$
(see the second inclusion in \eqref{eq:inclusions3}),
which is fibrewise $J_{t_j}$-holomorphic, to find a  fibrewise 
$J_{t_j}$-holomorphic function 
$\tilde f_j: \overline T_j \times \Omega_1\to\C$ which agrees 
with $f\circ \Phi_j^{-1}$
on $(\overline T_j\cap Q)\times \Omega_1$. The function 
$\tilde f_j \circ \Phi_j$ is then well-defined and  
$\Jscr$-holomorphic on $\overline T_j\times \Omega$ 
(see the first inclusion in \eqref{eq:inclusions3}), and 
it agrees with $f$ on $(\overline T_j \cap Q) \times \Omega$.
Choose a partition of unity $\{\chi_j\}_{j=1}^m$ on a neighbourhood
of $Q$ with $\supp \chi_j \subset T_j$. The function 
$f'=\sum_{j=1}^m \chi_j  (\tilde f_j \circ \Phi_j)$ has 
the desired properties. We now replace
$f$ by $(1-\xi) f+\xi f'$ where $\xi:T\to [0,1]$ is a continuous 
function with compact support contained 
in a small neighbourhood $Q'\supset Q$ such that $\xi=1$
on a neighbourhood of $Q$. 
This new function is $\Jscr$-holomorphic on $U\supset K$
and on $T'\times \Omega$, and it is close to the original 
function $f$ on $K$ if the neighbourhood $Q'\supset Q$ 
was chosen small enough. We apply the previously explained
construction to this new function, working on the complement of 
$Q$ in $T$ to find a $\Jscr$-holomorphic function
$F$ on a neighbourhood $U'$ of $L'$ which approximates $f$ on $K$
and agrees with $f$ on $(Q\times X)\cap U'$. This completes the 
induction step.

When $T$ is a locally compact and paracompact
Hausdorff space, the above proof gives an open locally finite
cover $\Tcal=\{T_j\}_j$ of $T$ (not necessarily countable) and
functions $F_j:\overline T_j\times X\to\C$ which approximate $f$ 
as closely as desired in the $\Cscr^{0,k}$ topology on 
$K\cap (\overline T_j\times X)$ and agree with $f$ on 
$(\overline T_j\cap Q)\times X$. Choosing a partition of unity
$\{\xi_j\}_j$ on $T$ subordinate to $\Tcal$ and setting
$F=\sum_j \xi_j F_j:T\times X\to \C$ gives functions satisfying the theorem. 
\end{proof}

We now turn to the Oka principle in Theorem \ref{th:Oka} where the 
target $Y$ is an arbitrary Oka manifold. We shall use the 
following special case of \cite[Lemma 6.3]{Forstneric2024Runge} 
which we restate using the notation of this paper. In this lemma, 
the Stein structure on $X$ is independent of the parameter $t$.

%
%
\begin{lemma}\label{lem:main}
Assume that $P''\subset\R^N$ is a neighbourhood retract
and $P_0\subset P_1\subset P \subset P'$ are compact subsets of $P''$, 
each contained in the interior of the next one.
Let $X$ be a Stein manifold, $\pi:\C^N\times X\to \C^N$ be the projection, 
and $K\subset \C^N\times X$ be a compact subset such that
$\pi(K)\subset P$ and the fibre $K_t=\{x\in X:(t,x)\in K\}$ 
is $\Oscr(X)$-convex for every $t\in P$.
(The fibre $K_t$ may be empty for some $t$.) 
Assume that $U$ is an open neighbourhood of $K$ 
in $P' \times X$, $Y$ is an Oka manifold endowed with a distance
function $\dist_Y$, and $f:P' \times X \to Y$ is a continuous map
such that for every $t\in P$ the map $f_t=f(t,\cdotp):X\to Y$
is holomorphic on $U_t=\{x\in X:(t,x)\in U\}$. 
Fix $\epsilon>0$. After shrinking the open set $U\supset K$
if necessary, there is a homotopy $f_s : P \times X \to Y$ $(s\in I=[0,1])$ 
satisfying the following conditions. 
\begin{enumerate}[\rm (a)]
\item $f_0=f|_{P\times X}$.
\item $f_s(t,\cdotp) : X\to Y$ is holomorphic on $U_t$
for every $s\in I$ and $t\in P$.
\item $\max_{(t,x)\in K} \dist_Y(f_s(t,x),f(t,x))<\epsilon$ for every $s\in I$. 
\item $f_s(t,\cdotp)=f(t,\cdotp)$ for all $t\in P\setminus P_1$ and $s\in I$.
\item The map $f_1(t,\cdotp):X\to Y$ is holomorphic for every $t$
in a neighbourhood of $P_0$.
\end{enumerate}
\end{lemma}

\begin{proof}[Proof of Theorem \ref{th:Oka}]
Let the integers $k\ge 1$ and $r\ge 2k+2n+9$ be as in 
Remark \ref{rem:Hamilton}. 
We shall follow \cite[proof of Theorem 1.6]{Forstneric2024Runge}, 
which treats the case when $X$ is a smooth open surface.
The adjustment we have to make is that the $J_t$-convex hull 
of a compact set in $X$ may now change with the parameter $t$. 
Unlike in the proof of Theorem \ref{th:OkaWeil}, 
we can not glue partial approximants by partitions of unity on $T$
since the target $Y$ is a manifold, so the problem is nonlinear.
Instead, we make all deformations by homotopies and use cut-off
functions in the parameters of the homotopy at every inductive step.

The conditions on $T$ imply that it is locally compact, $\sigma$-compact
and Hausdorff. Choose a normal exhaustion 
$T_1\subset T_2\subset \cdots\subset \bigcup_{j=1}^\infty T_j=T$ by 
compact sets (that is, each $T_j$ is contained in the interior of $T_{j+1}$).
Tameness of $\mathscr{J}$ provides an increasing sequence
$L_1\subset L_2 \subset \cdots\subset \bigcup_{j=1}^\infty L_j=X$
of compact sets forming a normal exhaustion of $X$ such that
for all $j=1,2,\ldots$ we have that 
\[
	K\cap (T_j\times X) \subset T_j\times L_j \quad \text{and}\quad  
	\wh{(T_{j}\times L_{j})}_{\!\!\!\Jscr} 
	\subset  T_{j}\times L_{j+1}.
\]

\begin{figure}[h]
  	\centering
	\includegraphics[scale=0.95]{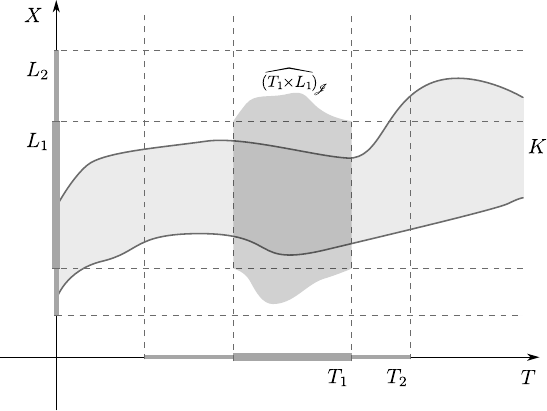}
 	\caption{An illustration of the choice of the set $L_2$.}
	\label{fig:L2}
\end{figure}
\noindent 

Define an increasing sequence of subsets 
$K=K^0\subset K^1\subset \cdots 
\subset \bigcup_{j=0}^\infty K^j=T\times X$ by 
\begin{equation}\label{eq:Kj}
	K^j = \wh{(T_j\!\times\! L_j)}_{\!\!\mathscr{J}} \cup K,
	\quad j=1,2,\ldots.
\end{equation}
Note that each $K^j$ is proper over $T$ and $\mathscr{J}$-convex.
Let $f^0=f:T\times X\to Y$ be the given map in the theorem
which is $\Jscr$-holomorphic on a neighbourhood of $K=K^0$
and on $Q\times X$.
We may assume that the distance function $\dist_Y$ is complete.
Let $\epsilon:T\to (0,+\infty)$ be the continuous function in the theorem.
We shall find a sequence of continuous maps
$f^j:T\times X\to Y$ and homotopies $f^j_s:T\times X\to Y$
$(s\in I=[0,1])$ satisfying the following conditions for every $j=1,2,\ldots$.
\begin{enumerate}[\rm (A)] 
\item $f^j$ is $\mathscr{J}$-holomorphic on a neighbourhood of 
the set $K^j$ \eqref{eq:Kj}. 
\item $f^j_0=f^{j-1}$ and $f^j_1=f^j$. 
\item $f^j_s$ is $\mathscr{J}$-holomorphic on a neighbourhood of $K^{j-1}$
for every $s\in I$, where the neighbourhood does not depend 
on $s\in I$.
\item $\max_{x\in K_t^{j-1}} \dist_Y (f^{j-1}(t,x),f^j_s(t,x))<2^{-j}\epsilon(t)$
for all $s\in I$ and $t\in T$. 
\item The homotopy $f^j_s(t,\cdotp)$, $s\in I$, is fixed for all $t\in Q$. 
\end{enumerate}
These conditions clearly imply that the homotopies $f^j_s$ $(j\in\N,\ s\in I)$ 
can be assembled into a single homotopy $f_s:T\times X\to Y$ $(s\in I)$ 
from the initial map $f_0=f=f^0$ to the limit $\mathscr{J}$-holomorphic 
map $f_1=F=\lim_{j\to\infty} f^j:T\times X\to Y$ (condition (iv))
such that for every $s\in I$, $f_s$ is $\mathscr{J}$-holomorphic 
on a neighbourhood of $K$ (condition (ii)), 
it approximates $f$ to precision $\epsilon$ on $K$ (condition (iii)), 
and the homotopy is fixed over $Q$ (condition (v)). 

Every step in the induction is of the same kind, so it suffices 
to show the initial step with $j=1$. This is accomplished by a finite 
induction which we now explain.

If the subset $Q\subset T$ in condition (v) is nonempty, 
we choose a small neighbourhood $Q_1\subset T$ of $Q$ and 
deform $f^0=f$ by a homotopy which is fixed for 
$t\in Q\cup (T\setminus Q_1)$ to another map $\tilde f^0:Z\to Y$ 
which approximates $f^0$ in the fine topology on $K^0$
such that $\tilde f^0(t,\cdotp)$ is holomorphic on a neighbourhood 
of $K^1_t$ for every $t$ in a closed neighbourhood $\wt Q \subset Q_1$
of $Q$ and the other properties of $f^0$ remain in place. 
This modification can be done in a similar way as in the last part 
of the proof of Theorem \ref{th:OkaWeil} but using the gluing technique
in \cite[Proposition 5.13.1]{Forstneric2017E} instead of  
partitions of unity. To simplify the notation, we replace
$f^0$ by $\tilde f^0$ and drop the tilde. Define the set 
\begin{equation}\label{eq:wtK0}
	\wt K^0 := \big[(\wt Q\times X)\cap K^1\big] \cup K^0 
	\subset T\times X.
\end{equation}
Note that $K^0 \subset \wt K^0 \subset K^1$, $\wt K^0$ is proper
over $T$ and $\Jscr$-convex, and $f^0$ is $\Jscr$-holomorphic
on a neighbourhood of $\wt K^0$. If $Q=\varnothing$,
we take $\wt Q=\varnothing$ and $\wt K^0=K^0$. 

Fix a point $t_0\in T_1$. Since $T_2$ is compact, there are 
a smoothly bounded strongly $J_{t_0}$-pseudoconvex 
domain $\Omega\Subset X$ and $J_{t_0}$-Stein 
domains $V, V'$ in $X$ such that 
\begin{equation}\label{eq:incl1}
   	\bigcup_{t\in T_2}K^1_t \Subset V\Subset V' \Subset \Omega. 
\end{equation}
The conditions on $T$ imply that there is a neighbourhood 
$P''\subset T$ of $t_0$
which is an ENR, so we may assume that 
$P''\subset \R^N\subset\C^N$ is a neighbourhood retract. 
Theorem \ref{th:Hamilton} gives a compact neighbourhood
$P'\subset T_2$ of $t_0$, contained in the interior of $P''$,  
and a map $\Phi: P'\times \Omega\to X$ of class $\Cscr^{0,k}$ 
and of the form $\Phi(t,x)=(t,\Phi_t(x))$ such that
$\Phi_t : \Omega \to \Phi_t(\Omega) \subset X$
is a $(J_t,J_{t_0})$-biholomorphism for every $t\in P'$ 
and $\Phi_{t_0} = \operatorname{Id}_\Omega$. 
Shrinking $P'$ around $t_0$ we may assume that for every 
$t\in P'$ we have
\begin{equation}\label{eq:incl2}
	U_t\subset\Omega,\quad 
	\Phi_t(\wt K^0_t) \subset V \subset V' \subset \Phi_t(\Omega),
	\quad 
	K^1_t\subset V \subset  \Phi_t^{-1}(V').
\end{equation}
(These are analogues of conditions \eqref{eq:inclusions}.)
Pick a compact neighbourhood $P\subset T$ of $t_0$, 
contained in the interior of $P'$, 
and consider the continuous family of maps 
\[
	f'_t :=  f_t\circ \Phi_t^{-1}: \Phi_t(\Omega) \to Y,\quad t\in P.
\]
Since the map $\Phi^{-1}_t: (\Phi_t(\Omega),J_{t_0}) \to (\Omega,J_t)$ 
is biholomorphic and $f_t$ is $J_t$-holomorphic on a neighbourhood
of $\wt K^0_t$, the map $f'_t$ is $J_{t_0}$-holomorphic on a 
neighbourhood of $\Phi_t(\wt K^0_t)$ for every $t\in P$
(see the second set of inclusions in \eqref{eq:incl2}). 
Pick a pair of smaller neighbourhoods $P_0\subset P_1\subset P$ of 
$t_0$, each contained in the interior of the next one. 
Lemma \ref{lem:main}, applied to the family of $J_{t_0}$-convex sets 
$\Phi_t(\wt K^0_t)$ in the $J_{t_0}$-Stein domain 
$V'\subset X$, gives a homotopy 
\[
	f'_{s,t}: V' \to Y\ \ \text{for $t\in P$ and $s\in I$} 
\]
such that $f'_{s,t}=f'_{0,t}=f'_t$ holds for $t\in P\setminus P_1$ and $s\in I$, 
the map $f'_{s,t}$ is $J_{t_0}$-holomorphic on a neighbourhood of
$\Phi_t(\wt K^0_t)$ and approximates $f'_t$ uniformly on 
$\Phi_t(\wt K^0_t)$ to arbitrary precision for all $t\in P$
and $s\in I$, and $f'_{1,t}$ is $J_{t_0}$-holomorphic on 
$V'$ for $t$ in a neighbourhood of $P_0$. 
By the third set of inclusions in \eqref{eq:incl2} 
we have that $\Phi_t(V)\subset V'$ for $t\in P$. It follows that the maps 
\begin{equation}\label{eq:ftb}
	f_{s,t}:=f'_{s,t} \circ \Phi_t: V \to Y\ \ 
	\text{for $s\in I$ and $t\in P$} 
\end{equation}
are $J_t$-holomorphic on a neighbourhood of $\wt K^0_t$,
$f_{s,t}$ approximates $f_t$ uniformly on $\wt K^0_t$ 
(and uniformly in $s\in I$) to arbitrary  precision, 
we have $f_{s,t}=f_{0,t}=f_t$ for $s\in I$ and $t\in P\setminus P_1$,
and the map $f^1_t:=f_{1,t}:V \to Y$ is 
$J_t$-holomorphic for all $t$ in a neighbourhood of $P_0$. 
We extend the family of homotopies to all $t\in T$ by setting 
$f_{s,t}=f_{0,t}=f_t$ for $t\in T\setminus P_1$ and $s\in I$.

Note that for $t\in P_1$ the map $f_{s,t}$ in \eqref{eq:ftb} is still 
defined only on $V\subset X$. In order to extend the 
homotopy to all of $X$ also for $t\in P_1$, choose a smooth 
cut-off function $\chi_1: X\to [0,1]$ such that $\chi_1=1$ 
in a neighbourhood of the compact set 
$\bigcup_{t\in P} K^1_t$ and $\supp \chi_1\subset V$. 
If the sets $Q\subset \wt Q$ are nonempty, 
we choose a second cut-off  function $\chi_2:T\to [0,1]$ such that 
$\chi_2=1$ on $T\setminus \wt Q$ and $\chi_2=0$ 
on a neighbourhood of $Q$. 
If $Q$ is empty we simply take $\chi_2=1$ on $T$. 
We can now extend the maps $f_{s,t}=f_s(t,\cdotp)$ 
in \eqref{eq:ftb} to all of $X$ without changing their values on a 
neighbourhood of $K^1$ by setting 
\[
	\tilde f_{s,t}(x) := f_{s\chi_1(x)\chi_2(t),t}(x) 
	\quad \text{for $t\in T$, $x\in X$, and $s\in I$}. 
\]	
For $t$ in a neighbourhood of $Q$ we have $\chi_2=0$ and hence 
$\tilde f_{s,t}=f_{0,t}=f_t$. 

Since $T_1$ is compact, we can find a finite family of triples
$P_0^j\subset P_1^j \subset P^j$ $(j=1,2,\ldots,m)$ of compact 
sets in $T$ such that $T_1\subset \bigcup_{j=1}^m P_0^j$ and
the construction described above can be made on each of these triples. 
The induction proceeds as follows. 
In the first step, we perform the procedure explained above on 
the first triple $(P^1_0,P^1_1,P^1)$ with the set $K^0$
and the map $g^0:=f^0=f$. We obtain a homotopy from 
$g^0$ to $g^1:T\times X \to Y$ 
such that every map in the homotopy is $\mathscr{J}$-holomorphic on
a neighbourhood of $K^0$, it approximates $g^0=f^0$ on $K^0$ to precision 
$\epsilon/2m$, and the homotopy is fixed for $t$ in a neighbourhood 
of $\overline {T\setminus P^1_1} \cup Q$.
The resulting map $g^1$ is $\mathscr{J}$-holomorphic 
on a neighbourhood of the compact $\Jscr$-convex set 
\[	
	S^1 := \big[(P^1_0\times X)\cap K^1\big] \cup  K^0 
	\subset T\times X. 
\]
Similary we define compact $\Jscr$-convex sets $S^\ell$ 
for $\ell=2,\dots, m$ by
\begin{equation*}
	S^\ell = \big[((P^1_0\cup\dots \cup  P^\ell_0)\times X)\cap K^1\big]
	\cup K^0  \subset T\times X.
\end{equation*} 

\begin{figure}[h]
  \centering
\includegraphics[scale=0.95]{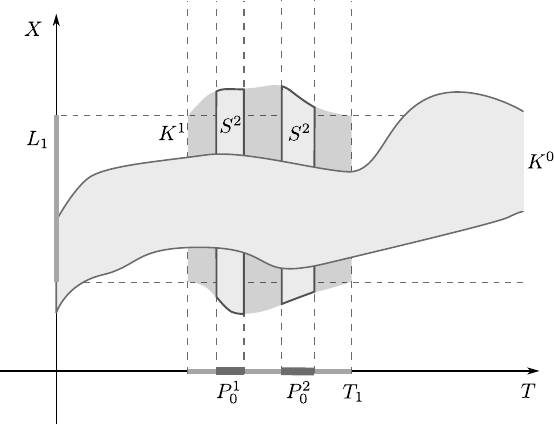}
  \caption{The set $S^2$. \label{fig:S2}}
\end{figure}

In step $\ell\in \{2,\dots, m\}$ the same argument 
is applied to the map $g^{\ell-1}$ 
on the triple $(P^\ell_0,P^\ell_1,P^\ell)$ with respect to the set $S^{\ell-1}$.
The resulting map $g^\ell: T\times X\to Y$ is $\mathscr{J}$-holomorphic 
on a neighbourhood of $S^\ell$. We also obtain a homotopy 
from $g^{\ell-1}$ to $g^\ell$ consisting of maps which are 
$\mathscr{J}$-holomorphic on a neighbourhood of $S^{\ell-1}$, they 
approximate $g^{\ell-1}$ on $S^{\ell-1}$ to precision 
$\epsilon/2m$, and the homotopy is fixed for $t$ in a neighbourhood of 
$\overline{T\setminus P^\ell_1} \cup Q$.
After $m$ steps we obtain a map $g^m:T\times X\to Y$ which is 
$\mathscr{J}$-holomorphic on a neighbourhood
of $S^m$, which contains $K^1$ (see \eqref{eq:Kj}). We define $f^1:=g^m$.
Furthermore, the homotopies between the subsequent maps 
$g^{\ell}$ and $g^{\ell+1}$ for $\ell=0,1,\ldots,m-1$ can be assembled 
into a homotopy $f^1_s$ $(s\in I)$ from the initial map $f^1_0 = f^0=g^0$ 
to $f^1_1=f^1=g^m$ such that $f^1_s$ satisfies (i)--(v) for $j=1$. 
This explains the inductive step and thereby concludes the proof.
\end{proof}

%
%
\section{The Oka--Weil theorem for 
sections of fibrewise holomorphic vector bundles} 
\label{sec:OkaWeil}

%
%

In this section, we assume that $T$ is a 
locally compact and paracompact Hausdorff space,
$X$ is a smooth manifold, and $\Jscr=\{J_t\}_{t\in T}$ 
is a tame family (see Def.\ \ref{def:tame}) of Stein structures on $X$ 
of class $\Cscr^{0,\infty}$. The main result of this section,
Theorem \ref{th:OkaWeilB}, is an Oka--Weil approximation
theorem for $\Jscr$-holomorphic 
sections of $\Jscr$-holomorphic vector bundles on $Z=T\times X$. 
It generalises Theorem \ref{th:OkaWeil}, which pertains to sections 
of a trivial vector bundle. Theorem \ref{th:OkaWeilB} will
be used in the following section to obtain global solvability
of the $\dibar$-equation on tame families of Stein manifolds;
see Theorem \ref{th:dibarGlobal}.

%
%
\begin{definition}\label{def:FHVB}
Let $T$, $X$ and $\Jscr$ be as above, and let $GL_r(\C)$
denote the complex Lie group of invertible $r\times r$ matrices.
A complex vector bundle $\pi:E \to Z=T\times X$
of rank $r$ is said to be $\Jscr$-holomorphic 
if it is defined on an open covering $\{U_i\}_i$ of $Z$
by a 1-cocycle $g=(g_{i,j})$ consisting of $\Jscr$-holomorphic maps
$g_{i,j}:U_{i,j}\to GL_r(\C)$.
\end{definition}

Explicitly, $E$ is obtained by gluing the trivial 
bundles $U_i\times \C^r$ on the overlaps $U_{i,j}=U_i\cap U_j$
by identifying a point $(z,v)\in U_j\times \C^r$ 
(where $z = (t,x) \in U_{i,j}$)
with $(z,g_{i,j}(z)v)\in U_i\times \C^r$.

%
%
\begin{theorem}
\label{th:OkaWeilB}
Let $T$, $X$ and $\Jscr$ be as above and 
$E\to Z=T\times X$ be a $\Jscr$-holomorphic vector bundle. 
Given a proper $\Jscr$-convex subset $K\subset Z$ (see \eqref{eq:whK2}), 
an open subset $U\subset Z$ containing $K$,  a closed subset $Q\subset T$,
a $\Jscr$-holomorphic section 
$f$ of $E$ over $U\cup(Q\times X)$, and an integer $k\in\Z_+$, we can approximate
$f$ in the fine $\Cscr^{0,k}$ topology on $K$ by $\Jscr$-holomorphic
sections $F:Z\to E$ such that $F=f$ on $Q\times X$. 
\end{theorem}

\begin{proof}
We shall assume that $T$ is compact and  $Q=\varnothing$;
the general case can be dealt with as in the proof of 
Theorem \ref{th:OkaWeil}, which pertains to sections of trivial bundles. 

Since $K$ is proper over the compact set $T$, it is compact.
It suffices to show that, given a $\Jscr$-convex subset
$L\subset Z$ with $K\subset L$, we can approximate $f$ in the  
$\Cscr^{0,k}$ topology on $K$ by $\Jscr$-holomorphic
sections of $E$ over a neighbourhood of $L$. The conclusion of the 
theorem then follows by induction with respect to an increasing sequence 
of $\Jscr$-convex subsets exhausting $Z$. 
Furthermore, since the problem is linear,
we may use partitions of unity on $T$.
This reduces the proof to the approximation problem for parameter 
values $t$ in a neighbourhood of a given point $t_0\in T$.

Fix $t_0\in T$ and smoothly bounded strongly 
$J_{t_0}$-pseudoconvex domains 
$\Omega_i\Subset X$ for $i=1,2,3$ such that 
$
	L_{t_0}\subset \Omega_1 \Subset \Omega_2\Subset \Omega_3.
$  
Theorem \ref{th:Hamilton} gives a compact neighbourhood $T_0\subset T$
of $t_0$ and a map $\Phi: T_0\times \Omega_3\to T_0\times X$ 
of class $\Cscr^{0,\infty}$ such that $\Phi(t,x)=(t,\Phi_t(x))$ and
$\Phi_t : \Omega_3 \to \Phi_t(\Omega_3) \subset X$ 
is a $(J_t,J_{t_0})$-biholomorphism for every $t\in T_0$, 
with $\Phi_{t_0} = \operatorname{Id}_\Omega$. 
Shrinking $T_0$ around $t_0$, we may assume that 
the following inclusions hold for every $t\in T_0$:
\begin{equation}\label{eq:inclusions2}
	L_t \subset \Omega_1 \subset \Phi_t^{-1}(\Omega_2),
	\qquad \Omega_2 \Subset \Phi_t(\Omega_3).
\end{equation}
Since the map $\Phi_t$ is $(J_t,J_{t_0})$-biholomorphic, the push-forward
vector bundle $\Phi_*(E|(T_0\times \Omega))$ is continuous in $t\in T_0$
and fibrewise $J_{t_0}$-holomorphic. Denote by $E'$ the restriction 
of this bundle to the domain 
$T_0\times \overline\Omega_2 \subset \Phi(T_0\times \Omega_3)$
(see \eqref{eq:inclusions2}). 
Its restriction $E'_t$ to the fibre $\{t\}\times \overline \Omega_2$
is a $J_{t_0}$-holomorphic vector bundle, smooth up to the 
boundary of $\Omega_2$ and depending continuously on $t\in T_0$.
Assuming as we may that $T_0$ is chosen small enough,
the stability result of Leiterer \cite[Theorem 2.7]{Leiterer1990} 
gives a family of $J_{t_0}$-holomorphic vector bundle isomorphisms 
over $\Omega_2$,
\begin{equation}\label{eq:Psit}
	\Psi_t: E'_t \stackrel{\cong}{\longrightarrow} E'_{t_0},\quad t\in T_0,
\end{equation}
smooth up to the boundary and depending continuously on $t\in T_0$. 
(The cited result is stated in terms of $J_{t_0}$-holomorphic transition 
cocycles $g^t=\{g^t_{i,j}\}$ for $E'_t$ for $t\in T_0$,
defined on a fixed open cover of $\overline \Omega_2$ and continuous
up to the boundary of the respective domains, and there are 
cohomological condition (i), (ii) on the endomorphism bundle
$\mathrm{Ad}(E_{t_0})$ of $E_{t_0}$. As explained in 
\cite[Remark 2.11]{Leiterer1990} and 
\cite[proof of Theorem 2.12]{Leiterer1990},
the two cohomology groups appearing in the hypothesis of 
\cite[Theorem 2.7]{Leiterer1990} vanish when the base is a 
compact strongly pseudoconvex domain with $\Cscr^2$ boundary.) 
The upshot is that the bundle $E'\to T_0\times \overline \Omega_2$ is 
fibrewise isomorphic to the trivial (independent of $t$) extension 
of the vector bundle $E'_{t_0}:=E'| (\{t_0\}\times \overline \Omega_2)$. 

Denote by $E_t$ the restriction of the initial vector bundle $E\to Z$ to the 
fibre over $t\in T$. We are given a continuous family of $J_t$-holomorphic 
sections $f_t:U_t\to E_t|U_t$, $t\in T$. For every $t\in T_0$, the map 
$\tilde f_t:=f_t\circ\Phi_t^{-1}$ is a $J_{t_0}$-holomorphic section of
the push-forward bundle $E'_t =(\Phi_t)_* E_t$ over the domain 
$\Phi_t(U_t)$, depending continuously on $t\in T_0$. By using the 
isomorphisms $\Psi_t$ in \eqref{eq:Psit}, we may consider
$\{\tilde f_t\}_{t\in T_0}$
as a family of $J_{t_0}$-holomorphic sections of $E'_{t_0}$
over the family of domains $\Phi_t(U_t)\supset \Phi_t(K_t)$.
(Here, $K_t$ is the fibre of the set $K$ in the theorem.)
Note that for every $t\in T_0$ the set $\Phi_t(K_t)$ is 
$J_{t_0}$-convex in $\Phi_t(\Omega_3)$, hence in 
$\Omega_2$; furthermore  $\Phi(K\cap (T_0\times X))$
is proper over $T_0$. 

By the parametric Oka--Weil theorem 
for sections of holomorphic vector bundles over Stein manifolds
(see \cite[Theorem 2.8.4]{Forstneric2017E}), we can approximate
$\tilde f_t$ in the $\Cscr^k$ topology on $\Phi_t(K_t)$, uniformly
in $t\in T_0$, by $J_{t_0}$-holomorphic sections $F'_t$ of 
the bundle $E'_{t_0}$ over the Stein domain $\Omega_2$.
(Approximation on variable fibres $\Phi_t(K_t)$ is reduced
to the case of constant fibres by the same technique
as in the proof of Theorem \ref{th:OkaWeil}, using a 
continuous partition of unity on $T_0$.)
Applying again the vector bundle isomorphisms \eqref{eq:Psit},
we may consider $F'_t$ as a $J_{t_0}$-holomorphic section of 
the bundle $E'_t$ over $\Omega_2$. Finally, $F_t:=F'_t\circ\Phi_t$ 
is a $J_t$-holomorphic section of the original bundle $E_t$ 
over the domain $\Phi_t^{-1}(\Omega_2)$, and these sections
depend continuously on $t\in T_0$. By \eqref{eq:inclusions2}
we have $L_t \subset \Omega_1\subset \Phi_t^{-1}(\Omega_2)$ 
for all $t\in T_0$, so $F(t,x)=(t,F_t(x))$ is a $\Jscr$-holomorphic
section of $E|(T_0\times \Omega_1)$ which 
approximates $f$ on $K\cap (T_0\times X)$. Note that 
$L\cap (T_0\times X) \subset T_0\times \Omega_1$.

Since $T$ is compact, this gives a finite open cover 
$\{W_j=T_j\times \Omega_j\}_j$ of $L$ such that $\Tcal=\{T_j\}_j$ 
is an open cover of $T$, and $\Jscr$-holomorphic sections $F_j$ of 
$E|W_j$ approximating $f$ in the $\Cscr^{0,k}$ topology on 
$K\cap W_j$ as closely as desired for every $j$. 
If $\{\chi_j\}_j$ is a continuous partition 
of unity on $T$ subordinate to $\Tcal$ then $F=\sum_j \chi_j F_j$ is a 
$\Jscr$-holomorphic section of $E$ over a neighbourhood 
of $L$ in $Z$ approximating $f$ in the fine $\Cscr^{0,k}$
topology on $K$. As explained at the beginning, 
this concludes the proof.
\end{proof}

%
%
\section{Global solution of the $\dibar$-equation on tame families 
of Stein manifolds}
\label{sec:dibar}

In this section, we assume that $T$ is a locally compact 
and paracompact Hausdorff space,
$X$ is a smooth manifold, and $\Jscr=\{J_t\}_{t\in T}$ 
is a tame family of Stein structures on $X$ (see Def.\ \ref{def:tame})
of class $\Cscr^{0,\infty}$. Write $Z=T\times X$.
Every fibre $Z_t=\{t\}\times X\cong X$ 
is endowed with the Stein structure $J_t$.
For each pair of integers $p\ge 0,\ q\ge 0$ we denote by 
$\Dscr^{p,q}(Z)$ the space of $(p,q)$-forms on the 
fibres $(X,J_t)$ of $Z$ of class $\Cscr^{0,\infty}$.
The following is a corollary to Theorems \ref{th:dibar} and \ref{th:OkaWeilB}.
A related result for $p=0,\ q=1$ on families of open Riemann surfaces 
is \cite[Corollary 1.2]{Forstneric2025dibar}.

%
%
\begin{theorem}\label{th:dibarGlobal}
Let $p\ge 0$ and $q\ge 1$. Given a 
family $\alpha=\{\alpha_t\}_{t\in T} \in \Dscr^{p,q}(Z)$ of 
smooth $(p,q)$-forms $\alpha_t\in \Dscr^{p,q}(X,J_t)$
with $\dibar_{J_t} \alpha_t =0$ for all $t\in T$, 
there exists $\beta=\{\beta_t\}_{t\in T} \in \Dscr^{p,q-1}(Z)$ 
satisfying
\begin{equation}\label{eq:dibar1}
	\dibar_{J_t} \beta_t = \alpha_t \ \ 
	\text{on $X$ for every $t\in T$.}
\end{equation}
\end{theorem}

In the sequel, we shall often write the equation \eqref{eq:dibar1}
in the form $\dibar_{\!\!\Jscr}\beta=\alpha$ on $Z=T\times X$. 

\begin{proof}
We begin by showing that the equation \eqref{eq:dibar1} is solvable
on a neighbourhood of any proper $\Jscr$-convex subset 
$K\subset Z$ (see \eqref{eq:whK2}).

Denote by $K_t$ the fibre of $K$ over $t\in T$.
Fix a point $t_0\in {T}$ and a smoothly bounded
strongly $J_{t_0}$-pseudoconvex neighbourhood $D \Subset X$ 
of $K_{t_0}$. By tameness of $\Jscr$ and
Lemma \ref{lem:spsh}, there is a neighbourhood 
$T_0\subset T$ of $t_0$ such that $D$ is a strongly 
$J_{t}$-pseudoconvex neighbourhood of $K_{t}$ for all $t\in T_0$. 
By Theorem \ref{th:dibar} (e), there exists 
$\beta \in \Dscr^{p,q-1}(T_0\times D)$ solving 
$\dibar_{\!\!\Jscr} \beta = \alpha$ on $T_0\times D$. 
In this way, we obtain an open locally finite cover 
$\Tcal=\{T_i\}_{i\in I}$ of $T$, smoothly bounded domains $D_i\Subset X$
such that $K \subset \bigcup_{i\in I} T_i\times D_i$, 
and solutions $\beta_i \in \Dscr^{p,q-1}(T_i\times D_i)$
to $\dibar_{\!\!\Jscr} \beta_i= \alpha$ on $T_i\times D_i$.
Let $\{\chi_i\}_{i\in I}$ be a partition of unity on $T$ 
subordinate to $\Tcal$. Then, 
$\beta=\sum_{i\in I}\chi_i \beta_{i}$ is a solution 
to \eqref{eq:dibar1} in a neighbourhood of $K$.

Choose an exhaustion $K^1\subset K^2\subset \cdots$ of $Z$ 
by proper $\Jscr$-convex sets (see \eqref{eq:whK2}). 
We shall inductively find solutions $\beta^j$ to \eqref{eq:dibar1} 
in neighbourhoods of $K^j$ such that $\beta^j$ approximates 
the solution $\beta^{j-1}$ from the previous
step on a neighbourhood of $K^{j-1}$ (if $q=1$), or agrees 
with it (if $q>1$).

Denote by $\Omega^{p}$ the sheaf of germs of 
$\Jscr$-holomorphic $(p,0)$-forms on the fibres of $Z=T\times X$. 
In particular, $\Omega^0=\Oscr$ is the sheaf
of germs of $\Jscr$-holomorphic functions on $Z$.
Since the complex structures $J_t\in \Jscr$ are smoothly compatible,
$\Omega^p$ is a subsheaf of the sheaf $\Escr^{p,0}$ of fibrewise 
smooth $(p,0)$-forms on the fibres of $Z$.
The elements $\beta\in \Dscr^{p,0}(Z)$ satisfying $\dibar_{\!\!\Jscr} \beta=0$ 
are precisely the global sections of $\Omega^p$ over $Z$. Equivalently, 
they are holomorphic sections of the $\Jscr$-holomorphic 
vector bundle on $Z$ (see Definition \ref{def:FHVB})
whose restriction to $Z_t=(X,J_t)$ is $\Lambda^p \, T^{*(1,0)}(X,J_t)$, 
the $p$-th exterior power of 
the $(1,0)$-cotangent bundle of $(X,J_t)$
(see Section \ref{sec:almostcomplex}). 

Let $\beta^j$ be as above, solving $\dibar_{\!\!\Jscr}\beta^j=\alpha$
on a neighbourhood of $K^j$ for $j=1,2,\ldots$. Then, 
\begin{equation}\label{eq:difference}
	 \dibar_{\!\!\Jscr} (\beta^j-\beta^{j-1})=0
	 \ \ \text{holds on a neighbourhood $U$ of $K^{j-1}$}. 
\end{equation}
If $q=1$, this means that $\beta^j-\beta^{j-1}$ is a 
section of the sheaf $\Omega^p$ on $U$.
By Theorem \ref{th:OkaWeilB}, we can approximate
it in the fine $\Cscr^{0,j}$ topology on $K^{j-1}$ by a global
section $\gamma$ of $\Omega^p$. Replacing 
$\beta^j$ by $\beta^j-\gamma$ ensures that $\beta^j$
solves $\dibar_{\!\!\Jscr}\beta^j=\alpha$
on a neighbourhood of $K^j$ and it approximates $\beta^{j-1}$ 
on $K^{j-1}$. Performing this construction inductively gives a  
sequence $\beta^j$ converging in the fine $\Cscr^{0,\infty}$-topology
to a solution $\beta\in \Dscr^{p,0}(Z)$ of 
the equation \eqref{eq:dibar1}.

Assume now that $q>1$. As explained earlier, 
\eqref{eq:difference} implies that 
$\beta^j-\beta^{j-1}=\dibar_{\!\!\Jscr}\gamma$
on a neighbourhood $U$ of $K^{j-1}$ for some
$\gamma\in \Dscr^{p,q-2}(U)$. Let $\chi:Z\to[0,1]$
be a function of class $\Cscr^{0,\infty}$ with $\supp(\chi)\subset U$ 
which equals $1$ in a neighbourhood of $K^{j-1}$.
Replacing $\beta^j$ by $\beta^j-\dibar_{\!\!\Jscr}(\chi \gamma)$
gives a solution to $\dibar_{\!\!\Jscr}\beta^j=\alpha$
in a neighbourhood of $K^{j}$ such that $\beta^j=\beta^{j-1}$
holds in a neighbourhood of $K^{j-1}$. Hence, the 
sequence $\beta^j$ is stationary and hence 
converges to a global solution of \eqref{eq:dibar1}.
\end{proof}

%
%
\begin{theorem}\label{th:Dolbeault}
(Assumptions as above.) 
$H^q(Z,\Omega^p)=0$ for all $q=1,2,\ldots$.
\end{theorem}

The groups $H^q(Z,\Omega^p)$ are classically called 
Dolbeault cohomology groups, although Dolbeault's opinion
was that they should in fact be called Grothendieck groups.

\begin{proof}
Let $\Escr^{p,q}$ denote the sheaf of fibrewise smooth $(p,q)$-forms 
on $Z=T\times X$ which are continuous in $t\in T$
(i.e.\ of class $\Cscr^{0,\infty}$).
Consider the sequence of homomorphisms of sheaves of abelian groups
\begin{equation}\label{eq:resolution}
	0\lra \Omega^p \longhookrightarrow \Escr^{p,0}
	\stackrel{d_0}{\lra} \Escr^{p,1} 
	\stackrel{d_1}{\lra} \Escr^{p,2} \stackrel{d_2}{\lra} 
	\cdots 
\end{equation}
where each $d_j$ is the $\dibar_{\!\!\Jscr}$ operator 
which equals $\dibar_{J_t}$ on $Z_t=(X,J_t)$ for every $t\in T$. 
By Theorem \ref{th:dibarGlobal} the sequence \eqref{eq:resolution} is exact.
All sheaves in \eqref{eq:resolution} except $\Omega^p$ 
are fine sheaves, 
so their cohomology groups of order $\ge 1$ vanish. 
(See e.g.\ \cite[Chapter VI]{GunningRossi1965} or 
\cite{Wells2008} for sheaf cohomology.) 
Hence, \eqref{eq:resolution} is an acyclic resolution of the 
sheaf $\Omega^p$. It follows that
\[
	H^q(Z,\Omega^p)= 
	\frac{
	\mathrm{Ker} \{d_{q}: \Gamma(Z,\Escr^{p,q}) \to 
	\Gamma(Z,\Escr^{p,q+1})\}}
	{\mathrm{Im} \{d_{q-1} : \Gamma(Z,\Escr^{p,q-1}) 
	\to \Gamma(Z,\Escr^{p,q})\}}
	= \frac{ 
	   \{\alpha\in \Dscr^{p,q}(Z): \dibar_{\!\!\Jscr} \alpha=0\}} 
	   {\{\dibar_{\!\!\Jscr}\beta: \beta\in \Dscr^{p,q-1}(Z)\}}.
\] 
Here, $\Gamma$ denotes the space of sections.
The group on the right hand side vanishes by Theorem \ref{th:dibarGlobal}.
\end{proof}

%
%
\section{The Oka principle for vector bundles on tame families of Stein manifolds} 
\label{sec:bundles}

Assume that $T$ is a topological space, $X$ is a smooth manifold, 
and $\Jscr=\{J_t\}_{t\in T}$ is a tame family of Stein structures on $X$. 
The notion of a $\Jscr$-holomorphic vector bundle on $Z=T\times X$ 
was introduced in Definition \ref{def:FHVB}. In this section, 
we prove the Oka principle for $\Jscr$-holomorphic vector bundles on 
tame families of Stein manifolds. 
We begin with line bundles. Denote by $\Pic(Z)$ the set of isomorphism 
classes of $\Jscr$-holomorphic line bundles on $Z=T\times X$. 
We have the following Oka principle which was proved
for line bundles on Stein manifolds (with $T$ a singleton)
by Oka \cite{Oka1939}.

%
%
\begin{theorem}\label{th:OPLB}
Assume that $T$ is a locally compact and paracompact Hausdorff space
and $\Jscr=\{J_t\}_{t\in T}$ is a tame family of class 
$\Cscr^{0,\infty}$ of Stein structures on a smooth manifold $X$. 
Then, every topological complex line bundle on $Z=T\times X$ 
is isomorphic to a $\Jscr$-holomorphic line bundle, 
and any two $\Jscr$-holomorphic line bundles on $Z$
which are topologically isomorphic are also
isomorphic as $\Jscr$-holomorphic line bundles.
Furthermore, $\Pic(Z)\cong H^2(Z,\Z)$.
\end{theorem} 

The proof of this result follows the standard cohomological 
argument for the exponential sheaf sequence on $Z$, 
using that $H^1(Z,\Oscr)=0$ and $H^2(Z,\Oscr)=0$
(see Theorem \ref{th:Dolbeault}) and 
$H^1(Z,\Oscr^*)=\Pic(Z)$. We refer to
\cite[Sect.\ 5.2]{Forstneric2017E} for the classical
case of line bundles on Stein manifolds,
and to \cite[Theorem 2.3]{Forstneric2025dibar} 
for line bundles on families of open Riemann surfaces.

For vector bundles of arbitrary rank, we have the following Oka principle.

%
%
\begin{theorem}\label{th:OPvectorbundles} 
Assume that $T$, $\Jscr=\{J_t\}_{t\in T}$ and $X$
are as in Theorem \ref{th:Oka}, with $\Jscr$ of class $\Cscr^{0,\infty}$.
Then, every topological vector bundle on $Z=T\times X$ is isomorphic 
to a $\Jscr$-holomorphic vector bundle, and every 
pair of $\Jscr$-holomorphic vector bundles which are topologically
isomorphic are also isomorphic as $\Jscr$-holomorphic vector bundles.
\end{theorem}

\begin{proof}
The proof of the first statement follows that of 
\cite[Theorem 2.4]{Forstneric2025dibar},
which gives an analogous result on families on open Riemann surfaces.
Let $Gr_r(\C^N)$ denote the Grassmann manifold of complex 
$r$-dimensional subspaces of $\C^N$, and let $\U\to Gr_r(\C^N)$
denote the universal bundle whose fibre over $\Lambda\in Gr_r(\C^N)$
consists of all vectors $v\in \Lambda\subset\C^N$.
Every topological vector bundle of rank $r$ on $Z$
is obtained as the pullback by a continuous map
$f:Z\to Gr_r(\C^N)$ of the universal bundle $\U$ for a sufficiently
big $N$; furthermore, homotopic maps induce isomorphic
vector bundles, and $\Jscr$-holomorphic maps induce  
$\Jscr$-holomorphic vector bundles. Since $Gr_r(\C^N)$
is a complex homogeneous manifold, and hence an 
Oka manifold, every continuous map $Z\to Gr_r(\C^N)$
is homotopic to a $\Jscr$-holomorphic map by 
Theorem \ref{th:Oka}. This proves the first part.

To prove the second statement, let $E\to Z$ and 
$E' \to Z$ be $\Jscr$-holomorphic vector bundles of rank $r$. 
There is an open cover $\{U_j\}_j$ of $Z$ 
and $\Jscr$-holomorphic vector bundle isomorphisms 
$\theta_j: E|{U_j} \to U_j\times\C^r$ and 
$\theta'_j: E'|{U_j}\to U_j\times\C^r$. Set $U_{i,j}=U_i\cap U_j$
and let
\[
    g_{i,j} : U_{i,j}  \to GL_r(\C), \quad g'_{i,j} : U_{i,j} \to GL_r(\C)
\]
denote the $\Jscr$-holomorphic transition maps of the two bundles, 
so that 
\[
    \theta_i \circ\theta_j^{-1}(z,v)= \bigl(z,g_{i,j}(z)v\bigr),
    \quad z\in U_{i,j},\ v\in\C^r,
\]
and likewise for $E'$. A complex vector bundle isomorphism 
$\Phi : E \to E'$ is given by a collection of complex vector bundle 
isomorphisms $\Phi_j: U_j\times\C^r\to U_j\times \C^r$ of the form
\[
    \Phi_j(z,v)=\bigl(z,\phi_j(z)v\bigr),
    \quad z\in U_j, \ v\in \C^r,
\]
with $\phi_j(z)\in GL_r(\C)$ for $z\in U_j$, 
satisfying the compatibility conditions
\begin{equation}\label{eq:bundle-P}
    \phi_i= g'_{i,j}\phi_j g_{i,j}^{-1} = g'_{i,j} \phi_j g_{j,i} 
    \quad {\rm on}\ U_{i,j}.
\end{equation}
Let $P=\delta(E,E') \to Z$ denote the $\Jscr$-holomorphic fibre bundle 
with fibre $GL_r(\C)$ and transition maps \eqref{eq:bundle-P},
so a collection of maps $\phi_j : U_j\to GL_r(\C)$ satisfying
\eqref{eq:bundle-P} is a section of $P$ over $Z$. 
Thus, complex vector bundle isomorphisms $E\to E'$ 
correspond to sections of $P\to Z$, with $\Jscr$-holomorphic
isomorphisms corresponding to $\Jscr$-holomorphic sections.
This reduces the problem to proving that every continuous section
$f:Z\to P$ is homotopic to a $\Jscr$-holomorphic section.

We proceed as in the proof of Theorem \ref{th:OkaWeilB}.
By Theorem \ref{th:Hamilton}, for every $t_0\in T$ and
smoothly bounded strongly pseudoconvex domain $\Omega\Subset X$ 
there are a neighbourhood $T_0\subset T$ of $t_0$ and a map 
$\Phi: T_0\times \Omega\to T_0\times X$ 
of class $\Cscr^{0,\infty}$ such that $\Phi(t,x)=(t,\Phi_t(x))$ and
$\Phi_t : \Omega \to \Phi_t(\Omega) \subset X$ 
is a $(J_t,J_{t_0})$-biholomorphism for every $t\in T_0$, 
with $\Phi_{t_0} = \operatorname{Id}_\Omega$. 
The push-forward bundles 
\[
	\wt E=\Phi_*(E|(T_0\times \Omega)),\quad 
	\wt E' = \Phi_*(E'|(T_0\times \Omega))
\]
are fibrewise $J_{t_0}$-holomorphic and depend continuously 
on $t\in T_0$. Choose a pair of strongly $J_{t_0}$-pseudoconvex domain 
$\Omega_1\Subset \Omega_2\Subset X$ such that 
\[
	T_0\times \overline \Omega_2\subset \Phi(T_0\times\Omega)
	\ \ \text{and}\ \ 
      T_0\times \overline \Omega_1\subset \Phi^{-1}(T_0\times\Omega_2).
\]
After shrinking $T_0$ around $t_0$, the stability 
theorem of Leiterer \cite[Theorem 2.7]{Leiterer1990} 
gives a family of $J_{t_0}$-holomorphic vector bundle isomorphisms 
$\Psi_t: \wt E_t \stackrel{\cong}{\longrightarrow} E_{t_0}$ 
over $\Omega_2$ (see \eqref{eq:Psit}) 
depending continuously on $t\in T_0$. 
We get similar isomorphisms 
$\Psi'_t: \wt E'_t \stackrel{\cong}{\longrightarrow} E'_{t_0}$ 
for the bundle $\wt E'$ over $T_0\times \Omega_2$. 
The upshot is that the vector bundles 
$\wt E|(T_0\times \Omega_2)$ and $\wt E'|(T_0\times \Omega_2)$ are 
fibrewise $J_{t_0}$-isomorphic to the trivial (independent of $t$) 
extensions of the vector bundles 
$E_{t_0}|\Omega_2$ and $E'_{t_0}|\Omega_2$, respectively.
In this local picture, a topological isomorphism
$E\to E'$ is given by a family of topological isomorphisms
$E_{t_0}|\Omega_2 \stackrel{\cong}{\longrightarrow} E'_{t_0}|\Omega_2$
depending continuously on $t\in T_0$, and a $\Jscr$-holomorphic 
isomorphism $E\to E'$ is given by a family of $J_{t_0}$-holomorphic isomorphisms. 
Such isomorphisms correspond to sections of a $J_{t_0}$-holomorphic
fibre bundle $H\to \Omega_2$ with fibre $GL_r(\C^N)$ defined as above,
see \eqref{eq:bundle-P}.

By the parametric Oka principle for sections of holomorphic
fibre bundles with Oka fibres over Stein manifolds, 
a family of topological sections of $H\to \Omega_2$ is isomorphic
to a family of holomorphic sections, with approximation on compact 
holomorphically convex subsets of $\Omega_2$.
Going back to the original vector bundles $E$, $E'$
and $P=\delta(E,E')$, we see that
any continuous section of $P$ is homotopic over 
$T_0\times \Omega_1$ to a $\Jscr$-holomorphic 
section, with approximation on a $\Jscr$-convex subset where 
the section is already holomorphic. The globalisation scheme
in the proof of Theorem \ref{th:Oka} then applies and shows that 
every continuous section of $P\to Z$ is homotopic to a 
$\Jscr$-holomorphic section.
\end{proof}

\section{Open problems}\label{sec:problems}

In this final section we collect some open problems for future
investigation. The first problem of technical nature is related to the stability
of canonical solutions of the $\dibar$-equation.

\begin{problem}\label{prob:dibar} 
Does Theorem \ref{th:dibar} have an analogue with smooth 
dependence of solutions on the parameter $t\in T$ when 
$T$ is a smooth manifold and the family of complex structures
$\Jscr=\{J_t\}_{t\in T}$ is smooth in $(t,x)\in T\times X$?
\end{problem}

%
%
We are not aware of results in the literature concerning Problem \ref{prob:dibar}, 
except when $X$ is a surface (see \cite{Forstneric2024Runge,Forstneric2025dibar}).
An affirmative answer would give a similar generalisation of 
the parametric Hamilton's theorem (see Theorem \ref{th:Hamilton}), 
and hence of all our main results. The corresponding 
analogue of Theorem \ref{th:OkaWeil} 
would show that if $\Jscr=\{J_t\}_{t\in T}$ is a smooth tame family of Stein structures 
then the manifold $Z=T\times X$, with the complex structure $J_t$ 
on $\{t\}\times X$ for every $t\in T$, is a {\em Cartan manifold} 
in the sense of Jurchescu \cite[Sect.\ 6]{Jurchescu1988RRMPA}; 
see also the discussion and references in \cite[Remark 1.2]{Forstneric2024Runge}. 
Cartan manifolds are analogues of Stein manifolds in 
the category of smooth CR manifolds with integrable
complex tangent subbundle. 
For real analytic Cartan manifolds with CR codimension one,
the function theory and the Oka principle for vector bundles 
were treated by Mongodi and Tomassini  
\cite{MongodiTomassini2016,MongodiTomassini2019}.

Our main result, Theorem \ref{th:Oka}, shows that 
tame families $\Jscr$ of smooth Stein structures on a given manifold $X$
admit many $\Jscr$-holomorphic maps to any Oka manifold. 
Theorem \ref{th:OkaWeil} gives a similar result for functions
with more general parameter spaces.
Which additional properties can these maps have?
The following problem is of particular interest; 
see \cite[Theorem 2.4.1]{Forstneric2017E}
for the summary of the classical results for Stein manifolds
and references to the original papers.

\begin{problem}\label{prob:embedding}
Assume that $T$, $X$, and $\Jscr$ are as in Theorem \ref{th:OkaWeil}.
Is there a $\Jscr$-holomorphic map $F:Z=T\times X\to \C^N$
for a suitable $N\in\N$ such that for every $t\in T$
the $J_t$-holomorphic map $F(t,\cdotp):X\to\C^N$ is proper, 
an immersion, an embedding? In particular, taking $N=2\dim_\C X+1$,
is there an $F$ such that  $F(t,\cdotp):X\to\C^N$ is a
proper $J_t$-holomorphic embedding for every $t\in T$?
\end{problem}

The Oka principle in Theorem \ref{th:Oka} only pertains to 
maps to Oka manifolds. In light of the classical results for 
a single Stein manifold (see \cite[Theorem 5.4.4]{Forstneric2017E}),
the following is a natural question.

\begin{problem}\label{prob:fibrebundles}
Let $T$, $X$ and $\Jscr$ be as in Theorem \ref{th:Oka},
and let $E\to Z=T\times X$ be a topological $\Jscr$-holomorphic 
fibre bundle with an Oka fibre. Does the Oka principle 
hold for sections $Z\to E$?
\end{problem}


We expect that this holds true, but the proof would require a 
suitable reworking of all basic tools used in the proof of the Oka
principle for a single Stein manifold; see \cite[Chap.\ 5]{Forstneric2017E}. 

\medskip 
\noindent {\bf Acknowledgements.} 
Forstneri\v c is supported by the European Union (ERC Advanced grant HPDR, 101053085) and grants P1-0291 and N1-0237 from ARIS, Republic of Slovenia. 
Sigurðardóttir is supported by a postdoctoral fellowship from 
the Institute of Mathematics, Physics and Mechanics, Ljubljana.
The first named author wishes to thank Yuta Kusakabe and Finnur L\'arusson 
for helpful discussions on some of the topics of this paper.




\end{document}